\documentclass[11pt]{article}

\usepackage{fullpage,amsmath, amssymb,color}
\usepackage{algorithm}
\usepackage{algorithmic}
\usepackage{graphicx}
\usepackage{caption}
\usepackage{subfigure}
\usepackage{array}
\newcommand{\BlackBox}{\rule{1.5ex}{1.5ex}}  
\newenvironment{proof}{\par\noindent{\bf Proof\ }}{\hfill\BlackBox\\[2mm]}

\makeatletter
\newcommand*{\rom}[1]{\expandafter\@slowromancap\romannumeral #1@}
\makeatother
\usepackage{pifont}
\newcommand{\cmark}{\ding{51}}%
\newcommand{\xmark}{\ding{55}}%

\newcommand{\ve}[2]{\langle #1 ,  #2 \rangle}   
\newcommand{\eqdef}{\stackrel{\text{def}}{=}}
\newcommand{\R}{\mathbb{R}}

\newcommand{\Exp}{\mathbb{E}}
\newcommand{\Prob}{\mathbb{P}}
\newcommand{\calP}{\mathcal{P}}

\def\<#1,#2>{\langle #1,#2\rangle}

\DeclareMathOperator{\dom}{dom}

\newtheorem{theorem}{Theorem}
\newtheorem{lemma}[theorem]{Lemma}
\newtheorem{corollary}[theorem]{Corollary}
\newtheorem{proposition}[theorem]{Proposition}
\newtheorem{assumption}[theorem]{Assumption}

\title{Randomized Dual Coordinate Ascent with Arbitrary Sampling
}

\author{Zheng Qu\footnote{School of Mathematics, The University of Edinburgh, United Kingdom.} \qquad  Peter Richt\'{a}rik\footnote{School of Mathematics, The University of Edinburgh, United Kingdom.} \qquad  Tong Zhang\footnote{Department of Statistics, Rutgers University, New Jersey, USA and Big Data Lab, Baidu Inc,  China.  \qquad \qquad {\bf Acknowledgments:} The first two authors would like to  acknowledge support from the EPSRC  Grant EP/K02325X/1,
{\em Accelerated Coordinate Descent Methods for Big Data Optimization}.}}

\begin{document}
\maketitle

\begin{abstract}
We study the problem of minimizing the average of a large number of smooth convex functions penalized  with a strongly convex regularizer. We propose and analyze a novel primal-dual method (Quartz) which at every iteration samples and updates a random subset of the dual variables, chosen according to an {\em arbitrary distribution}.  In contrast to typical analysis, we directly bound the decrease of the primal-dual error (in expectation), without the need to first analyze the dual error.  Depending on the choice of the sampling, we obtain efficient serial, parallel and distributed variants of the method. In the serial case, our bounds match the best known bounds for SDCA (both with  uniform and importance sampling). With standard mini-batching, our bounds predict initial data-independent speedup as well {as \em additional data-driven speedup} which depends on spectral and sparsity properties of the data. We calculate theoretical speedup factors  and find that they are excellent predictors of actual speedup in practice. Moreover, we illustrate that it is possible to design an efficient {\em mini-batch importance} sampling. The distributed variant of Quartz is the first distributed SDCA-like method with an analysis for non-separable data.

\end{abstract}

\section{Introduction}

In this paper we consider a primal-dual pair of structured convex optimization problems which has in several variants of varying degrees of generality attracted a lot of attention in the past few years in the machine learning and optimization communities \cite{Hsieh:2008:DCD:1390156.1390208,Lin:2008:DCDM, SDCA, ProxSDCA, pegasos2, ASDCA, IProx-SDCA}.








\subsection{The problem}

Let $A_1,\dots,A_n$ be a collection of $d$-by-$m$ real matrices and $\phi_1, \dots, \phi_n$ be $1/\gamma$-smooth convex  functions from $\R^m$ to $\R$, where $\gamma>0$.
Further, let $g:\R^d\rightarrow \R$ be a $1$-strongly convex function and $\lambda>0$ a regularization parameter. We are interested in solving the following {\em primal} problem:
\begin{equation}\label{eq:primal}\min_{w=(w_1,\dots,w_d)\in \R^d}\;\; \left[ P(w) \eqdef \frac{1}{n}\sum_{i=1}^n \phi_i(A_i^\top w) + \lambda g(w)\right].\end{equation}

In the machine learning context, matrices $\{A_i\}$ are interpreted as examples/samples,  $w$ is a (linear) predictor, function $\phi_i$ is the loss incurred by the predictor on example $A_i$, $g$ is a regularizer, $\lambda$ is a regularization parameter and \eqref{eq:primal}  is the {\em regularized empirical risk minimization} problem. However, above problem has many other applications outside machine learning. In this paper we are especially interested in problems where $n$ is very big (millions, billions), and much larger than $d$. This is often the case in {\em  big data} applications.

Let $g^*:\R^d \to \R$ be the convex conjugate\footnote{In this paper, the convex (Fenchel) conjugate of a function $\xi:\R^k\to \R$ is the function $\xi^*:\R^k\to \R$ defined by $\xi^*(u) = \sup_{\|s\|=1} \{s^\top u - \xi(s)\}$, where $\|\cdot\|$ is the L2 norm.} of $g$ and for each $i$, let $\phi_i^*:\R^m\to \R$ be the convex conjugate of $\phi_i$. Associated with the {\em primal problem} \eqref{eq:primal} is the Fenchel {\em dual problem}:
\begin{equation}\label{eq:dual}\max_{\alpha=(\alpha_1,\dots,\alpha_n) \in \R^{N}=\R^{nm}}\;\; \left[D(\alpha)
\eqdef - f(\alpha)   -\psi(\alpha)\right],\end{equation}
where $\alpha = (\alpha_1,\dots,\alpha_n)\in \R^N=\R^{nm}$ is obtained by stacking dual variables (blocks) $\alpha_i\in \R^m$, $i=1,\dots,n$, on top of each other and  functions $f$ and $\psi$ are defined by
\begin{align}\label{a-defoff}
& f(\alpha)\eqdef \lambda g^*\left(\frac{1}{\lambda n}\sum_{i=1}^n A_i\alpha_i\right),\\\label{a-defofpsi}
&\psi(\alpha)\eqdef \frac{1}{n}\sum_{i=1}^n \phi_i^*(-\alpha_i).
\end{align}
Note that $f$ is convex and smooth and $\psi$ is strongly convex and block separable.

\subsection{Contributions} We now briefly list the main contributions of this work.

\paragraph{Quartz.} We propose a new algorithm, which we call Quartz\footnote{Strange as it may seem, this algorithm name appeared to one of the authors of this paper in a dream. According to Wikipedia: ``Quartz is the second most abundant mineral in the Earth's continental crust. There are many different varieties of quartz, several of which are semi-precious gemstones.''  Our method also comes in many variants. It later came as a surprise to the  authors that the name could be interpreted as QU And Richt{\' a}rik and Tong Zhang. Whether the subconscious mind of the sleeping coauthor who dreamed up the name knew about this connection or not is not known.}, for simultaneously solving the primal~\eqref{eq:primal} and  dual~\eqref{eq:dual} problems. On the dual side, at  each iteration our method selects and updates a {\em random subset (sampling)}  $\hat{S}\subseteq \{1,\dots,n\}$ of the dual variables/blocks. We assume that these sets are i.i.d. throughout the iterations. However,  {\em we do not impose any additional assumptions on the distribution} apart from the necessary requirement that each block $i \in [n]$ needs to be chosen with a positive probability: $p_i\eqdef\Prob(i\in \hat{S})>0$. Quartz is the first SDCA-like method analyzed for an {\em arbitrary sampling}.  The dual updates are then used to perform an update to the primal variable $w$ and the process is repeated. Our primal updates are different (less aggressive) from those used in SDCA \cite{SDCA} and Prox-SDCA \cite{ProxSDCA}.

\paragraph{Main result.} We prove that starting from an initial pair $(w^0,\alpha^0)$, Quartz finds a pair $(w,\alpha)$ for which $P(w)-D(\alpha)\leq \epsilon$ (in expectation) in at most
\begin{equation}\label{eq:complexity-intro} 
\max_i \left( \frac{1}{p_i} + \frac{v_i}{p_i \lambda \gamma n}\right)  \log \left( \frac{P(w^0)-D(\alpha^0)}{\epsilon} \right)\end{equation}
iterations. The parameters $v_1,\dots,v_n$ are assumed to satisfy the following ESO (expected separable overapproximation) inequality:
\begin{equation} \label{eq:ESOfirst}  \textstyle
\Exp_{\hat{S}} \left[ \left\| \sum_{i\in \hat{S} } A_i h_i \right\|^2 \right] \leq   \sum_{i=1}^n p_i v_i \|h_i\|^2.
\end{equation} 
Moreover, the parameters are needed to run the method (they determine stepsizes), and hence it is critical that they can be cheaply computed before the method starts. As we will show, for many samplings of interest this can be done in time required to read the data $\{A_i\}$. We wish to point out that \eqref{eq:ESOfirst}  always holds for {\em some} parameters $\{v_i\}$. Indeed,   the left hand side is a quadratic function of $h$ and hence the inequality holds for large-enough $v_i$. Having said that, the size of these parameters directly influences the complexity, and hence one would want to obtain as tight bounds as possible.

\paragraph{Arbitrary sampling.} As described above, Quartz uses an {\em arbitrary sampling} for picking the dual variables to be updated in each iteration. To the best of our knowledge, only a single paper exists in the literature where a stochastic  method using an arbitrary sampling was analyzed: the NSync method of Richt\'{a}rik and Tak\'{a}\v{c} \cite{NSync} (for unconstrained minimization of a strongly convex function). Assumption \eqref{eq:ESOfirst} was for  the first time introduced there  (in a more general form; we are using it here  in the special case of a quadratic function). However, NSync is not a primal-dual method. Besides NSync, the closest works to ours in terms of the generality of the sampling are the PCDM algorithm of Richt\'{a}rik and Tak\'{a}\v{c} \cite{PCDM}, SPCDM method of Fercoq and Richt\'{a}rik \cite{SPCDM} and the APPROX method of Fercoq and Richt\'{a}rik \cite{APPROX}. All these are randomized coordinate descent methods, and all  were analyzed for arbitrary {\em uniform} samplings (i.e., samplings satisfying $\Prob(i\in \hat{S})=\Prob(i'\in \hat{S})$ for all $i,i'\in[n]$). Again, none of these methods were analyzed in a primal-dual framework.

\paragraph{Direct primal-dual analysis.} Virtually all methods for solving \eqref{eq:primal} by performing stochastic steps in the dual \eqref{eq:dual}, such as SDCA \cite{SDCA}, SDCA for SVM dual \cite{pegasos2}, ProxSDCA \cite{ProxSDCA}, ASDCA \cite{ASDCA} and APCG \cite{APCG}, are analyzed by first establishing dual convergence and then proving that the  duality gap is bounded by the dual  residual. The SPDC method of Zhang and Xiao \cite{SPDC}, which is a stochastic coordinate update variant of the Chambolle-Pock method \cite{ChambolePock}, is an exception.  Our analysis is novel, and {\em directly primal-dual} in nature. As a result, our proof is more direct, and the logarithmic term in our bound has a simpler form.

\paragraph{Flexibility: many important variants.}  Our method is very flexible: by specializing it to specific samplings, we obtain numerous variants, some similar (but not identical) to existing methods in the literature, and some very new and of significance to big data optimization. 

\begin{itemize}
\item \textbf{Serial uniform sampling.} If $\hat{S}$ always picks a single block, uniformly at random ($p_i=1/n$), then the dual updates of Quartz are similar to  those of SDCA \cite{SDCA} and Prox-SDCA \cite{ProxSDCA}.  The leading term in the complexity bound \eqref{eq:complexity-intro}  becomes $n+ \max_i \lambda_{\text{max}}(A_i^\top A_i)/( \lambda \gamma)$, which matches the bounds obtained  in these papers. However, our logarithmic term is simpler.

\item \textbf{Serial optimal sampling (importance sampling).} If $\hat{S}$ always picks a single block, with $p_i$ chosen so as to minimize the complexity bound \eqref{eq:complexity-intro},  we obtain the same {\em  importance sampling} as that recently used in the IProx-SDCA method \cite{IProx-SDCA}. Our bound becomes $n+ (\tfrac{1}{n}\sum_i \lambda_{\text{max}}(A_i^\top A_i))/( \lambda \gamma)$, which matches the bound in \cite{IProx-SDCA}. Again, our logarithmic term is better. 

\item \textbf{$\tau$-nice sampling.} If we now let $\hat{S}$ be a random subset of $[n]$ of size $\tau$ chosen uniformly at random (this sampling is called $\tau$-nice in \cite{PCDM}), we obtain a mini-batch (parallel) variant of Quartz. There are only a handful of primal-dual stochastic methods which use mini-batching. The first such method was a mini-batch version of SDCA specialized to training $L2$-regularized linear SVMs with hinge loss \cite{pegasos2}.  Besides this,  two accelerated mini-batch methods have been recently proposed: ASDCA of Shalev-Shwartz and Zhang \cite{ASDCA} and SPDC of Zhang and Xiao \cite{SPDC}. The  complexity bound of Quartz specialized to the $\tau$-nice sampling is different, and despite Quartz not being an accelerated method, and can be better in certain regimes (we will do a detailed comparison in Section~\ref{sec:main_result}).

\item \textbf{Distributed sampling.} To the best of our knowledge, no other samplings than those described above were used in  stochastic primal-dual methods. However, there are many additional interesting samplings proposed  for randomized coordinate descent, but never applied to the primal-dual framework.   For instance, we can use the {\em distributed sampling} which led to the development of the Hydra algorithm \cite{Hydra} (distributed coordinate descent) and its accelerated variant  Hydra$^2$ (Hydra squared) \cite{Hydra2}. Using this sampling, Quartz can be efficiently implemented in a distributed environment (partition the examples across the nodes of a cluster, and let each node in each iteration update a random subset of variables corresponding to the examples it owns). 

\item \textbf{Product sampling.} We  describe a novel sampling, which we call {\em product sampling}, that can be {\em  both  non-serial and non-uniform}. This is the first time such a sampling has been described and and a SDCA-like method using it analyzed. For suitable data (if  the examples can be partitioned into several groups no two of which share a feature), this sampling can lead to linear or nearly linear speedup when compared to the serial uniform sampling.

\item \textbf{Other samplings.}  While we develop the analysis of Quartz for an arbitrary sampling, we do not compute the ESO parameters $\{v_i\}$ for any other samplings in this paper. However, there  are several other interesting choices. We refer the reader to \cite{PCDM} and \cite{NSync} for further examples of uniform and non-uniform samplings, respectively. All that must be done for any new $\hat{S}$ is to find parameters $v_i$ for which \eqref{eq:ESOfirst} holds, and the complexity of the new variant of Quartz is given by \eqref{eq:complexity-intro}.
\end{itemize}

\paragraph{Further data-driven speedup.} Existing mini-batch stochastic primal-dual methods achieve linear speedup up to a certain mini-batch size which depends on $n, \lambda $ and $\gamma$. Quartz obtains this data-independent speedup, but also obtains {\em further data-driven speedup}. This is caused by the fact that Quartz uses more aggressive dual stepsizes, informed by the data through the ESO parameters $\{v_i\}$. The smaller these constants, the better speedup. For instance, we will show that higher data  sparsity leads to smaller $\{v_i\}$ and hence to better speedup. To illustrate this, consider the $\tau$-nice sampling (hence, $p_i=\tau/n$ for all $i$) and  the extreme case of perfectly sparse data (each feature $j\in [d]$  appearing in a single example $A_i$). Then \eqref{eq:ESOfirst} holds with $v_i=\lambda_{\text{max}}(A_i^\top A_i)$ for all $i$, and hence the leading term in \eqref{eq:complexity-intro} becomes $n/\tau + \max_i \lambda_{\text{max}}(A_i^\top A_i)/(\gamma \lambda \tau)$,  predicting {\em perfect speedup} in the mini-batch size $\tau$. We derive {\em theoretical speedup factors} and show that these are excellent predictors of actual behavior of the method in an implementation. This was previously observed for the PCDM method \cite{PCDM} (which is not primal-dual). 

\paragraph{Quartz vs purely primal and purely dual methods.} 
In the special case when  $\hat{S}$ is the serial uniform sampling, the complexity of Quartz is similar to the bounds recently obtained by several purely primal stochastic and semi-stochastic gradient methods (all having reduced variance of the gradient estimate) such as SAG \cite{SAG}, SVRG \cite{SVRG}, S2GD \cite{S2GD},  SAGA \cite{SAGA}, mS2GD \cite{mS2GD} and MISO \cite{MISO}.
In the case of serial optimal sampling, relevant purely primal methods with similar guarantees are ProxSVRG \cite{proxSVRG} and S2CD  \cite{S2CD}. A mini-batch primal method, mS2GD, was analyzed in \cite{mS2GD}, achieving a similar bound to Quartz specialized to the $\tau$-nice sampling. Purely dual (stochastic coordinate descent) methods with similar bounds to Quartz for both the serial uniform and serial optimal sampling, for problems of varying similarity and generality when compared  to \eqref{eq:dual}, include   SCD \cite{ShalevTewari09}, RCDM \cite{Nesterov:2010RCDM}, UCDC/RCDC \cite{UCDC}, ICD   \cite{ICD} and RCD \cite{Necoara:rcdm-coupled}. These methods were then generalized to the $\tau$-nice sampling in SHOTGUN \cite{shotgun}, further generalized to arbitrary uniform samplings in PCDM \cite{PCDM}, SPCDM \cite{SPCDM}, APPROX \cite{APPROX} (which is an accelerated method) and to arbitrary (even nonuniform) samplings in NSync \cite{NSync}. Another accelerated method, BOOM, was proposed in \cite{mukherjee2013parallel}. Distributed randomized coordinate descent methods with purely dual analysis include Hydra \cite{Hydra} and Hydra$^2$ \cite{Hydra2} (accelerated variant of Hydra). Quartz specialized to the distributed sampling achieves the same rate as Hydra, but for both the primal and dual problems simultaneously. 
 

\paragraph{General problem.} We consider the  problem \eqref{eq:primal} (and consequently, the associated dual) in a rather general form; most existing primal-dual methods focus on the case when $g$ is a quadratic (e.g., \cite{SDCA, ASDCA}) or $m=1$  (e.g., \cite{SPDC}). Lower bounds for a variant of problem \eqref{eq:primal} were recently established by Agarwal and Bottou \cite{AgarwalBottou14}.


\subsection{Outline}

In Section~\ref{sec:alg} we describe the algorithm and show that it admits  a natural interpretation in terms of Fenchel duality. We also outline the similarities and differences of the primal and dual update steps with SDCA-like methods. In Section~\ref{sec:ESO} we show how parameters $\{v_i\}$ satisfying the ESO inequality \eqref{eq:ESOfirst} can be computed for several selected samplings. We then proceed to Section~\ref{sec:main_result} where we state the main result, specialize it to some of the samplings discussed in Section~\ref{sec:ESO}. Sections~\ref{sec:MAIN-tau-nice} and \ref{MAIN:distributed} deal with Quartz specialized to the $\tau$-nice and distributed sampling, respectively. We also give detailed comparison of our results with existing results for related primal-dual stochastic methods existing in the literature, and analyze theoretical speedup factors. We then   provide the proof of the main complexity result in Section~\ref{sec:proofs}. In Section~\ref{sec:experiments} we perform numerical experiments on the problem of training $L_2$-regularized linear support vector machine with square and smoothed hinge loss with real datasets. Finally, in Section~\ref{sec:conclusion} we conclude.

\section{The Quartz Algorithm} \label{sec:alg}

In this section we describe our method (Algorithm~\ref{algo:primaldual}).

\subsection{Preliminaries}

The most important parameter of  Quartz is a  random sampling $\hat{S}$ of the dual variables $[n]=\{1,2,\dots,n\}$. That is, $\hat{S}$  is a random subset of $[n]$, or more precisely, a random set-valued mapping with values being the subsets of $[n]$.  In order to guarantee that each block (dual variable) has a chance to get updated by the method, we necessarily need to make the following assumption.

\begin{assumption}[Proper sampling] \label{ass-proper} $\hat{S}$ is a proper sampling. That is, \begin{equation}
\label{eq:p_i}p_i\eqdef \Prob(i\in \hat{S})>0, \qquad i\in [n].\end{equation}
\end{assumption}

However, we shall not make any other assumption on $\hat{S}$. Prior to running the algorithm, we compute positive constants $v_1,\dots,v_n$ satisfying \eqref{eq:ESOfirst}---such constants always exist---as these are used to define the stepsize parameter $\theta$ used throughout:
\begin{align}
\theta=\min_i \frac{p_i \lambda \gamma n}{v_i+\lambda \gamma n}.
\end{align}

We shall show how this parameter can be computed for various samplings in Section~\ref{sec:ESO}. Let us now formalize the notions of $(1/\gamma)$-smoothness and strong convexity.

\begin{assumption}[Loss] \label{ass:loss} For each $i\in [n]$, the loss function $\phi_i:\R^m\to \R$ is convex, differentiable and has  $(1/\gamma)$-Lipschitz continuous gradient with respect to the L2 norm, where $\gamma$ is a positive constant:
\[ \|\nabla \phi_i(x)-\nabla \phi_i(y)\| \leq \frac{1}{\gamma} \|x-y\|,  \quad x,y\in \R^m.\]
For brevity, the last property is often called $(1/\gamma)$-smoothness.
\end{assumption}

It follows that $\phi_i^*$ is  $\gamma$-strongly convex.

\begin{assumption}[Regularizer] \label{ass:reg} The regularizer $g:\R^d\to \R$ is $1$-strongly convex. That is, \[g(w)\geq g(w') +\ve{\nabla g(w')}{w-w'} + \tfrac{1}{2}\|w-w'\|^2, \quad  w,w'\in \R^d,\]
where $\nabla g(w')$ is a subgradient of $g$ at $w'$.
\end{assumption}

It follows that $g^*$ is $1$-smooth.

\subsection{Description of the method}

Quartz starts with an initial  pair of primal and dual vectors $(w^0,\alpha^0)$. Given  $w^{t-1}$ and $\alpha^{t-1}$, the method maintains the vector
\begin{equation}\label{jibs79gss}\bar{\alpha}^{t-1} = \frac{1}{\lambda n} \sum_{i=1}^n A_i \alpha_i^{t-1}.\end{equation}
Initially this is computed from scratch, and subsequently  it is maintained in an efficient manner at the end of each iteration.

\begin{algorithm}
\begin{algorithmic}[ht]
\STATE \textbf{Parameters}: proper random sampling $\hat S$ and a positive vector $v\in \R^n$ 
\STATE \textbf{Initialization}: Choose $\alpha^0 \in \R^{N}$ and $w^0\in \R^d$
\\~\hspace{2.5cm} Set  $p_i=\Prob(i\in \hat S)$,
$ \theta=\displaystyle\min_i \tfrac{p_i \lambda \gamma n}{ v_i+\lambda \gamma n}$ and
 $\bar \alpha^0=\tfrac{1}{\lambda n}\sum_{i=1}^n A_i \alpha^0_i $
\FOR{ $t \geq 1$}
\STATE $w^t=(1-\theta)w^{t-1}+\theta \nabla g^*(\bar \alpha^{t-1})$
\STATE $\alpha^{t}=\alpha^{t-1}$
\STATE Generate a random set $S_t\subseteq [n]$, following the distribution of $\hat S$
\FOR{$i\in S_t$}
\STATE Calculate $\Delta\alpha_i^t$ using one of the following options:
\STATE $\quad$\textbf{Option} \rom{1} :

$ \quad\Delta\alpha^{t}_{i}=\arg\max_{\Delta\in \R^m}
\left[-\phi_i^*(-(\alpha_i^{t-1}+\Delta ))-\nabla g^*(\bar\alpha^{t-1})^\top A_i \Delta-\frac{v_i\|\Delta\| ^2}{2\lambda n}\right]
 $
\STATE $\quad$\textbf{Option} \rom{2} :

$\quad \Delta\alpha^{t}_{i}=-\theta p_i^{-1}\alpha_{i}^{t-1}-\theta p_i^{-1}\nabla \phi_{i}(A_{i}^\top w^t)$

 \STATE $\alpha_i^t=\alpha_i^{t-1}+\Delta\alpha_i^t$
\ENDFOR
\STATE $\bar \alpha^{t}=\bar \alpha^{t-1}+{(\lambda n)^{-1}}\sum_{i\in S_t} A_{i}\Delta\alpha_i^{t}$
\ENDFOR 
\STATE \textbf{Output:} $w^t$, $\alpha^t$
\end{algorithmic}
\caption{Quartz}
\label{algo:primaldual}
\end{algorithm}

Let us now describe how the vectors $w^t$ and $\alpha^{t}$ are computed. Quartz first updates the primal vector $w^t$ by setting it to a {\em convex combination} of the previous value $w^{t-1}$ and $\nabla g^*(\bar{\alpha}^{t-1})$:
\begin{equation}\label{eq:98h9s8h00}w^t  =(1-\theta)w^{t-1}+\theta \nabla g^*(\bar \alpha^{t-1}).\end{equation}

We then proceed to select, and subsequently update, a random subset $S_t\subseteq [n]$ of the dual variables,  independently from the sets drawn in previous iterations, and following the distribution of $\hat{S}$. Clearly, there are many ways in which the distribution of $\hat{S}$ can be chosen, leading the numerous variants of Quartz. We shall describe some of them in Section~\ref{sec:ESO}. We allow two options for the actual computation of the dual updates. Once the dual variables are updated, the vector $\bar{\alpha}^t$ is updated in an efficient manner so that  \eqref{jibs79gss} holds. The entire process is repeated.

\paragraph{Fenchel duality interpretation.} Quartz has a natural interpretation in terms of Fenchel duality. Fix a primal-dual pair of vectors  $(w,\alpha)\in \R^d\times \R^N$ and define 
$\bar{\alpha} = \frac{1}{\lambda n}\sum_{i=1}^n A_i\alpha_i$. The duality gap for the pair $(w,\alpha)$ can be decomposed as follows:
\begin{eqnarray*}
P(w) - D(\alpha) 
&\overset{\eqref{eq:primal}+\eqref{eq:dual}}{=}&   \lambda \left( g(w)  + g^*\left(\bar{\alpha}\right)\right) + \frac{1}{n}\sum_{i=1}^n \phi_i(A_i^\top w) +   \phi_i^*(-\alpha_i) \\
&=& \lambda (\underbrace{g(w)  + g^*\left(\bar{\alpha}\right)- \left\langle w, \bar{\alpha} \right\rangle  }_{GAP_{g}(w,\alpha)}) + \frac{1}{n}\sum_{i=1}^n \underbrace{\phi_i(A_i^\top w) +   \phi_i^*(-\alpha_i)  + \ve{A_i^\top w}{\alpha_i}}_{GAP_{\phi_i}(w,\alpha_i)}.
 \end{eqnarray*}
By Fenchel-Young inequality, $GAP_g(w,\alpha)\geq 0$ and $GAP_{\phi_i}(w,\alpha_i)\geq 0$ for all $i$, which proves weak duality for the problems \eqref{eq:primal} and \eqref{eq:dual}, i.e., $P(w)\geq D(\alpha)$. The pair $(w,\alpha)$ is optimal when both $GAP_{g}$ and $GAP_{\phi_i}$ for all $i$ are zero. It is known that this happens precisely when the following optimality conditions hold:
\begin{equation}\label{eq:98y98hss}w = \nabla g^*(\bar{\alpha}),\end{equation}
\begin{equation}\label{eq:s899sns09u} \alpha_i = -\nabla \phi_i(A_i^\top w), \quad \forall i\in [n].\end{equation}

We will now interpret the primal and dual steps of Quartz in terms of the above discussion. At iteration $t$ we first set the primal variable $w^t$ to a convex combination of its current value $w^{t-1}$ and  a value that would set $GAP_g$ to zero: see \eqref{eq:98h9s8h00}. Hence, our primal update is not as aggressive as that of Prox-SDCA. This is followed by adjusting the dual variables corresponding to a randomly chosen set of examples $S_t$. 
Under Option II, for each example $i\in S_t$, the $i$-th dual variable $\alpha_i^t$ 
is set to a convex combination of its current value 
$\alpha_i^{t-1}$ and the value that would set $GAP_{\phi_i}$ to zero: \[\alpha_i^t = \left(1-\frac{\theta}{p_i}\right) \alpha_i^{t-1} + \frac{\theta}{p_i}\left(-\nabla \phi_i(A_i^\top w^t)\right).\]

\paragraph{Quartz vs Prox-SDCA.} In the special case when $\hat{S}$ is the serial uniform sampling (i.e., $p_i=1/n$ for all $i\in[n]$), Quartz can be compared to  
 Proximal Stochastic Dual Coordinate Ascent (Prox-SDCA)~\cite{ASDCA,SDCA}.  Indeed,  if  Option \rom{1} is always used in Quartz,  then the dual update of $\alpha^t$ in Quartz is exactly the same as the dual update of Prox-SDCA (using Option \rom{1}).
 In this case, the difference between our method and Prox-SDCA lies in the update of the primal variable $w^t$: while Quartz performs the update \eqref{eq:98h9s8h00}, Prox-SDCA  (see also~\cite{NIPSdistributedSDCA,cocoa}) performs the more aggressive  update $w^t=\nabla g^*(\bar \alpha^{t-1})$.

\section{Expected Separable Overapproximation} \label{sec:ESO}

For the sake of brevity, it will be convenient to establish some notation.  Let $A=[A_1,\dots, A_n]\in \R^{d\times N}=
\R^{d\times nm}$ be the block matrix with blocks  $A_i \in \R^{d\times m}$. Further, let $A_{ji}$ be the $j$-th row of $A_i$. Likewise, for $h\in \R^N$ we will write $h=(h_1,\dots,h_n)$, where $h_i\in \R^m$, so that  $Ah = \sum_{i=1}^n A_i h_i$.
For a vector of positive weights $w\in \R^n$, we define a weighted Euclidean norm in $\R^N$ by
\begin{equation}\label{eq:i9s89s8hsnorm}
\|h\|_w^2 \eqdef \sum_{i=1}^n w_i \|h_i\|^2,
\end{equation}
where $\|\cdot\|$ is the standard Euclidean norm on $\R^m$. For $S\subset[n]\eqdef \{1,\dots,n\}$ and  $h\in\R^{N}$ we use the notation $h_{[S]}$ to denote 
the vector in $\R^N$ coinciding with $h$ for blocks $i\in S$ and zero elsewhere: 
$$
(h_{[S]})_i=\left\{\begin{array}{ll}
           h_i,& ~~\mathrm{if~} i\in S,\\
0,&~~\mathrm{otherwise.}
          \end{array}\right.
$$
With this notation, we have
\begin{equation} \label{eq:is9hssdsd} Ah_{[S]} = \sum_{i\in S} A_i h_i.\end{equation}


As mentioned before, in our analysis we require that the random sampling $\hat S$ and 
the positive vector $v\in\R^ n$ used in Quartz satisfy inequality  \eqref{eq:ESOfirst}. We shall now formalize this as an assumption, using the compact notation established above.

\begin{assumption}[ESO]\label{ass-ESO}  The following inequality holds for all $h\in\R^{N}$:
\begin{align}\label{a-ESO}\Exp[\| A h_{[\hat S]}\|^2] \leq   \|h\|_{p\cdot v}^2,\end{align}
where $p=(p_1,\dots,p_n)$ is defined in \eqref{eq:p_i},  $v=(v_1,\dots,v_n)>0$ and $p\cdot v = (p_1v_1,\dots,p_n v_n)\in \R^n$.
\end{assumption}

Note that for any proper  sampling $\hat S$, there must {\em exist} vector $v>0$ satisfying Assumption~\ref{ass-ESO}. Hence, this  is an assumption that such a vector $v$ is {\em readily available}. Indeed,  the term on the left is a finite average of convex quadratic functions of $h$, and hence is a convex quadratic. Moreover, we can write
$$
\Exp [\|Ah_{[\hat{S}]}\|^2] = \Exp[h_{[\hat S]}^\top A^\top A h_{[\hat S]}]=h^\top \left(P \circ A^\top A\right)  h,
$$
where $\circ$ denotes the Hadamard  (component-wise) product of matrices and $P\in \R^{N \times N}$ is an $n$-by-$n$ block matrix with block $(i,j)$  equal to $\Prob(i\in\hat S,j\in\hat S)1_m$, with  $1_m$ being the $m$-by-$m$ matrix of all ones. Hence~\eqref{a-ESO} merely means to upper bound
the matrix $P\circ A^\top A$ by an $n$-by-$n$ block diagonal matrix $D=D_{p,v}$, the $i$-th block of which is equal to $p_i v_i I_m$ with $I_m$ being the $m$-by-$m$ identity matrix. There is an infinite number of ways how this can be done (in theory). Indeed, for any proper sampling $\hat{S}$ and {\em any} positive $w\in \R^n$,  \eqref{a-ESO} holds with $v=tw$, where 
\[t = \lambda_{\text{max}} \left(D_{p,w}^{-1/2}(P\circ A^T A) D_{p,w}^{-1/2}\right),\]
since then $P\circ A^\top A \preceq t D_{p,w} = D_{p,v}.$

In practice, and especially in the big data setting when $n$ is very large, computing $v$ by solving an eigenvalue problem with an $N\times N$ matrix (recall that $N=nm$) will be either inefficient or impossible. It is therefore important that a ``good'' (i.e., small), albeit perhaps suboptimal $v$ can be identified {\em cheaply.} In all the cases we consider in this paper, the identification of $v$ can be done during the time the data is being read; or in  time roughly equal to a single pass through the data matrix $A$.

In the special case of  uniform\footnote{A sampling $\hat{S}$ is uniform if $p_i=p_j$ for all $i,j$. It is easy to see that then, necessarily, $p_i=\Exp[|\hat{S}|]/n$ for all $i$. The ESO inequality studied in \cite{PCDM} is of the form: $\Exp[\xi(x+h_{[\hat{S}]})] \leq \xi(x) + \tfrac{\Exp[|\hat{S}|]}{n}\left(\ve{\nabla \xi(x)}{h} + \tfrac{1}{2}\|h\|_v^2\right)$. In the case of uniform sampling, $x=0$ and $\xi(h) = \tfrac{1}{2}\|Ah\|^2$, we recover \eqref{a-ESO}.} samplings but for arbitrary smooth  functions (and not just quadratics; which is all we need here), inequality \eqref{a-ESO} was introduced and studied  by Richt\'{a}rik and Tak\'{a}\v{c} \cite{PCDM}, in the context of complexity analysis of (non primal-dual) parallel block coordinate descent  methods. A variant of ESO for arbitrary (possibly nonuniform) samplings was introduced in \cite{NSync}; and to the best of our knowledge that is the only work  analyzing a stochastic coordinate descent method which uses an arbitrary sampling. However, NSync is not a primal-dual method and applies to a different problem (unconstrained minimization of a smooth strongly convex function). Besides \cite{PCDM, NSync}, ESO inequalities were further studied  in~\cite{pegasos2, DQA,SPCDM,Hydra,APPROX,Hydra2,mS2GD}.

\subsection{Serial samplings}

The most studied sampling in literature on stochastic optimization is the \textit{serial sampling}, which corresponds to the selection of a single block $i \in [n]$. That is, $|\hat{S}| = 1$ with probability 1. The name ``serial'' is pointing to the fact that a method using such a sampling will typically be a serial (as opposed to being parallel) method;  updating a single block (dual variable) at a time.

A serial sampling is uniquely characterized by the vector of probabilities $p=(p_1,\dots,p_n)$, where $p_i$ is defined by \eqref{eq:p_i}. It turns out that we can  find a vector $v>0$ for which \eqref{a-ESO} holds for {\em any} serial sampling, {\em independently of its distribution} given by $p$. 

\begin{lemma}\label{l-serialv}
 If $\hat{S}$ is a serial sampling (i.e., if $|\hat S|=1$ with probability 1), then Assumption~\ref{ass-ESO} is satisfied for 
\begin{align}\label{a-vi}
v_i=\lambda_{\max} (A_i^\top A_i),\enspace i\in[n].
\end{align}
\end{lemma}
\begin{proof}
Note that for any $h\in \R^N$,
\[\Exp [\|Ah_{[\hat{S}]}\|^2] = \sum_{i=1}^n p_i \|A h_{[\{i\}]}\|^2 \overset{\eqref{eq:is9hssdsd}}{=} \sum_{i=1}^n p_i (h_i A_i^\top A_i h_i) \leq \sum_{i=1}^n p_i \lambda_{\max}(A_i^\top A_i) \|h_i\|^2 \overset{\eqref{eq:i9s89s8hsnorm}}{=} \|h\|_{p\cdot v}^2.\]
\end{proof}

Note that $v_i$ is the largest eigenvalue of an $m$-by-$m$ matrix. If $m$ is relatively small (and in many machine learning applications one has $m=1$; as examples are usually vectors and not matrices), then the cost of computing $v_i$ is small. If $m=1$, then  $v_i$ is simply the squared Euclidean norm of the vector $A_i$, and hence one can compute all of these parameters in one pass through the data (e.g.,  during loading to memory).

\subsection{Parallel ($\tau$-nice) sampling}\label{subsec-taunince}

We now consider  $\hat{S}$ which selects subsets of $[n]$ of cardinality $\tau$, uniformly at random. In the terminology established in \cite{PCDM}, such $\hat{S}$ is called $\tau$-nice. This sampling satisfies $p_i=p_j$ for all $i,j\in [n]$; and hence it is uniform. 

This sampling is well suited for parallel computing. Indeed, Quartz could be implemented as follows. If we have $\tau$ processors available, then  at the beginning of  iteration $t$ we can assign each block (dual variable) in  $S_t$ to a dedicated processor. The processor assigned to $i$ would then compute $\Delta\alpha_i^t$ and apply the update. If all processors have  fast access to   the memory where all the data is stored, as is the case in a shared-memory multicore workstation, then this way of assigning workload to the individual processors does not cause any major problems.  Depending on the particular computer architecture and the size $m$ of the blocks (which will influence processing time), it may be more efficient to chose $\tau$ to be a multiple of the number of processors available, in which case in each iteration every processor updates more than one block.

The following lemma gives a closed-form formula for parameters $\{v_i\}$ for which the ESO inequality holds.

\begin{lemma}[compare with~\cite{APPROX}]\label{prop-ESO}
  If $\hat S$ is a $\tau$-nice sampling, then Assumption~\ref{ass-ESO} is satisfied for
\begin{align}\label{a-vtaunice}
v_i= \lambda_{\max}\left(\sum_{j=1}^d \left(1+\frac{(\omega_j-1)(\tau-1)}{n-1}\right) A_{ji}^\top A_{ji}\right),\enspace i\in[n],
\end{align}
where for each $j\in[d]$, $\omega_j$ is the number of nonzero blocks in the $j$-th row of matrix $A$, i.e.,
\begin{align}\label{a-omegaj}
\omega_j \eqdef |\{i\in[n]: A_{ji}\neq 0\}|,\qquad j\in[d].
\end{align}
 \end{lemma}
 \begin{proof}
In the $m=1$ case the result follows from Theorem 1 in~\cite{APPROX}. Extension to the $m>1$ case is straightforward.
\end{proof}

Note that $v_i$ is the largest eigenvalue of an $m$-by-$m$ matrix which is formed as the sum of $d$ rank-one matrices. The formation of all of these $n$ matrices takes time proportional to the number of nonzeros in $A$ (if the data is stored in a sparse format). Constants $\{\omega_j\}$ can be computed by scanning the data once (e.g., during loading-to-memory phase). Finally, one must compute $n$ eigenvalue problems for matrices of size $m\times m$. In most applications, $m=1$, so there is no more work to be done. If $m>1$, the cost of computing these eigenvalues would be small.

While for $\tau=1$ it was easy to find parameters $\{v_i\}$ for any sampling (and hence, as we will see, it will be easy to find an optimal sampling), this is not the case in the $\tau>1$ case. The task is in general a difficult optimization problem.  For some work in this direction we refer the reader to \cite{NSync}.

\subsection{Product sampling}\label{subsec-products}
In this section we give an example of a sampling $\hat S$ which can be both non-uniform and non-serial  (i.e., for which  $\Prob(|\hat{S}|=1)\neq 1$). We make the following \textit{group separability assumption}:  there is a partition $X_1,\dots, X_{\tau}$ of $[n]$ according to 
which the examples $\{A_i\}$
 can be partitioned into $\tau$ groups such that no feature is shared by any two examples belonging to 
different groups. 

Consider the following example with $m=1, n=5$ and $d=4$:

\[A = [A_1,A_2,A_3,A_4,A_5] = \left(\begin{matrix} 0 & 0 & 6 & 4 & 9\\ 
0 & 3 & 0 & 0 & 0\\
0 & 0 & 3 & 0 & 1\\
1 & 8 & 0 & 0 & 0\\
\end{matrix}\right)\]
If we choose $\tau=2$ and $X_1 = \{1,2\}, X_2=\{3,4,5\}$, then no row of $A$ has a nonzero in both a column belonging to $X_1$ and a column belonging to $X_2$.

With each $i\in[n]$ we now associate $l_i  \in [\tau]$ such that $i\in X_{l_i}$ and
define:
$$
{\cal S} \eqdef  X_{1}\times \cdots \times X_{\tau}.
$$
The \textit{product sampling}
 $\hat{S}$ is obtained by choosing $S \in {\cal S}$, uniformly at random; that is, via: \begin{equation}\label{eq:si9s8hs}\Prob(\hat{S}=S) = \frac{1}{|{\cal S}|} = \frac{1}{\prod_{l=1}^\tau |X_l|}, \quad S \in {\cal S}. \end{equation}
Then $\hat{S}$ is proper and 
\begin{align}\label{a-pips}p_i \eqdef \Prob(i\in \hat{S}) =   \frac{\prod_{l \neq l_i}|X_l|}{|{\cal S}|} \overset{\eqref{eq:si9s8hs}}{=} \frac{1}{|X_{l_i}|}, \qquad i \in [n].\end{align}
Hence the sampling is nonuniform as long as not all of the 
sets $X_l$ have the same cardinality. 
We next show that the product sampling $\hat S$ defined as above  allows
 the same stepsize parameter $v_i$ as the serial uniform sampling.
\begin{lemma}\label{l-pips}
  Under the group separability assumption,  Assumption~\ref{ass-ESO} is satisfied for  the product sampling $\hat S$ and
\begin{align*}
v_i=\lambda_{\max} (A_i^\top A_i),\quad i\in[n].
\end{align*}
\end{lemma}
\begin{proof}
For each $j\in[d]$, denote by $A_{j:}$  the $j$-th row of the matrix $A$ and  $\Omega_j$ the column index set of nonzero blocks in $A_{j:}$:
$\Omega_j\eqdef\{i\in[n]: A_{ji}\neq 0\}.$
For each $l\in[\tau]$, define:
\begin{align}\label{a-defJl}
J_l\eqdef\{j\in[d]: \Omega_j \subset X_l\}.
\end{align}

In words, $J_l$ is the set of features  associated with  the examples in $X_l$. By the group separability assumption, $J_1,\dots,J_\tau$ forms a partition of $[d]$, namely,
 \begin{align}\label{a-cupJi}
\bigcup_{l=1}^{\tau} J_l=[d];\enspace J_k\cap J_l=\emptyset,\enspace \forall k\neq l\in[\tau] .
\end{align} 
Thus,
\begin{align}\label{a-AtopA}
A^\top A=\sum_{j=1}^d A_{j:}^\top A_{j:}\overset{\eqref{a-cupJi}}{=}\sum_{l=1}^\tau 
\sum_{j\in J_l} A_{j:}^\top A_{j:}.
\end{align}
 Now fix $l\in[\tau]$ and $j\in J_l$. For any $h \in \R^ N$ we have:
 \begin{align*}
 \Exp[h_{[\hat S]} A_{j:}^\top A_{j:} h_{[\hat S]}]&
=\sum_{i,i'\in[n]} 
h_i^\top A_{ji}^\top A_{ji'} h_{i'} \Prob(i\in \hat S,i'\in\hat S)
=\sum_{i,i'\in \Omega_j} 
h_i^\top A_{ji}^\top A_{ji'} h_{i'} \Prob(i\in \hat S,i'\in\hat S).
 \end{align*}
Since $X_1,\dots,X_{\tau}$ forms a partition 
of $[n]$, then any two indexes belonging to the same subset $X_l$ will never be selected simultaneously in $\hat S$, i.e.,
$$
\Prob(i\in \hat S, i'\in \hat S)=\left\{
\begin{array}{ll}
p_i & \mathrm{if~~} i=i' \\
0 &  \mathrm{if~~} i\neq i'
\end{array}
\right. ,\enspace \enspace \forall i,i' \in X_l.
$$
Therefore,
\begin{align}\label{a-prpt2}
 \Exp[h_{[\hat S]} A_{j:}^\top A_{j:} h_{[\hat S]}]&=\sum_{i\in \Omega_j} 
h_i^\top A_{ji}^\top A_{ji} h_{i} p_i=\sum_{i=1}^ n 
h_i^\top A_{ji}^\top A_{ji} h_{i} p_i.
 \end{align}
It follows from~\eqref{a-AtopA} and~\eqref{a-prpt2} that:
\begin{align}\label{a-prps1}
\Exp[\|Ah_{[\hat S]}\|^2]&=\Exp[h_{[\hat S]} A^\top A h_{[\hat S]}]=
\sum_{l=1}^\tau 
\sum_{j\in J_l}\Exp[h_{[\hat S]} A_{j:}^\top A_{j:} h_{[\hat S]}]=\sum_{l=1}^\tau \sum_{j\in J_l}\sum_{i=1}^n 
h_i^\top A_{ji}^\top A_{ji} h_{i} p_i.
 \end{align}
Hence,
$\Exp[\|Ah_{[\hat S]}\|^2]
\overset{\eqref{a-cupJi}}{=}\sum_{j=1}^d \sum_{i=1}^n  h_i^\top A_{ji}^\top A_{ji} h_{i} p_i 
\leq \sum_{i=1}^n \lambda_{\max}(A_i^\top A_i)h_i^ \top h_i p_i =\|h\|_{p\cdot v}^2.$
\end{proof}

\subsection{Distributed sampling} \label{sec:ESO-distributed}

We now describe a sampling which is particularly suitable for a {\em distributed implementation of Quartz}. This sampling was first proposed in~\cite{Hydra} and later used in \cite{Hydra2}, where the distributed coordinate descent algorithm Hydra and its accelerated variant Hydra$^2$ were proposed and analyzed, respectively. Both methods were shown to be able to scale up to huge problem sizes (tests were performed on problem sizes of several TB; and up 50 billion dual variables in size).

Consider a distributed computing environment with $c$ nodes/computers. For simplicity, assume that $n$ is an integer multiple of $c$ 
and let the blocks $\{1,2,\dots,n\}$   be  partitioned into $c$ sets of equal size: ${\cal P}_1$, ${\cal P}_2$, \dots, ${\cal P}_c$. We assign partition ${\cal P}_l$ to node $l$. The data  $A_1,\dots,A_n$ and the dual variables  (blocks) $\alpha_1,\dots,\alpha_n$ are
partitioned accordingly and stored on the respective nodes.

At each iteration, all nodes $l\in\{1,\dots,c\}$ in parallel pick a  subset $\hat S_l$ of $\tau$ dual variables from those they own, i.e., from $\calP_l$, uniformly at random. That is, each node locally  performs a $\tau$-nice sampling, independently from the other nodes. Node $l$ computes the updates to the dual variables $\alpha_i$ corresponding to $i\in S_l$, and locally stores them.
Hence, in a single distributed iteration, Quartz updates the dual variables belonging to the set  $ \hat S \eqdef \cup_{l=1}^c \hat S_l$.  This defines a sampling, 
which we will call  $(c,\tau)$-\textit{distributed sampling}.
 
Of course, there are other important considerations pertaining to the distributed implementation  of Quartz, but we do not discuss them here as the focus of this section is on the sampling. However, it is possible to design a distributed communication protocol for the update of the primal variable.
 
 The following result gives a formula for admissible parameters $\{v_i\}$.
 
\begin{lemma}[compare with~\cite{Hydra2}]\label{prop-distriESO}
 If $\hat S$ is a $(c,\tau)$-distributed sampling,  then
 Assumption~\ref{ass-ESO} is satisfied for
 \begin{align}\label{a-bDiibetaj}
v_i= \lambda_{\max}\left( \sum_{j=1}^d \left(1+\frac{(\tau-1)(\omega_j-1)}{\max\left\{\frac{n}{c}-1,1\right\}}+ \left(\frac{\tau c}{n} - \frac{\tau-1}{ \max \{ \frac{n}{c}-1,1\} }\right) \frac{\omega_j'-1}{\omega_j'}\omega_j\right) A_{ji}^\top A_{ji}\right), \quad i \in[n],
\end{align}
where $\omega_j$ is the number of nonzero blocks in the $j$-th row of the matrix $A$, as defined previously in~\eqref{a-omegaj}, and
$\omega'_j$ is the number of partitions "active" at row $j$ of $A$, more precisely,
\begin{align}\label{a-wjprime}
\omega'_j \eqdef |\{l\in [c]: \{i\in \calP_l: A_{ji}\neq 0\}\neq \emptyset\}|, \qquad  j\in[d].
\end{align}
 \end{lemma}
 \begin{proof}
 When $m=1$, the result is equivalent to Theorem 4.1 in~\cite{Hydra2}.  The extension to blocks ($m>1$) is straightforward.
 \end{proof}

Lemma~\ref{prop-ESO} is a special case of Lemma~\ref{prop-distriESO} when
only a single node ($c=1$) is used, in which case $\omega_j'=1$ for all $j\in[d]$. Lemma~\ref{prop-distriESO} also improves the constants $\{v_i\}$ derived in~\cite{Hydra}, where instead of  $\omega_j$ and $\omega_j'$ in~\eqref{a-bDiibetaj} one has
 $\max_j \omega_j$ and $\max_j \omega_j'$.

 Lemma~\ref{prop-distriESO} is expressed in terms of certain sparsity parameters associated with the data ($\{\omega_j\}$) and the partitioning ($\{\omega_j'\}$). However, it is possible to derive alternative ESO results for the $(c,\tau)$-distributed sampling. For instance, one can instead express the parameters $\{v_j\}$ without any sparsity assumptions, using only spectral properties of the data only. We have not included these results here, but in the $m=1$ case such results have been derived in \cite{Hydra2}. It is possible to adopt them to the $m=1$ case as we have done it with Lemma~\ref{prop-distriESO}.

\section{Main Result} \label{sec:main_result}

The complexity of our method is given by the following theorem.

\begin{theorem}[Main Result] \label{th-mainth}
Let Assumption \ref{ass:loss} ($\phi_i$ are $(1/\gamma)$-smooth) and Assumption~\ref{ass:reg} ($g$ is 1-strongly convex) be satisfied.  Let $\hat{S}$ be a proper sampling (Assumption~\ref{ass-proper}) and   $v_1,\dots,v_n$ be positive scalars satisfying  Assumption~\ref{ass-ESO}. Then the sequence of primal and dual variables
 $\{w^t,\alpha^t\}_{t\geq 0}$ of Quartz (Algorithm~\ref{algo:primaldual}) satisfies:
\begin{equation}\label{eq-con}
\Exp[P(w^t)-D(\alpha^t)]\leq (1-\theta)^t (P(w^0)-D(\alpha^0)),
\end{equation}
where \begin{align}\label{a-theta}
\theta=\min_i \frac{p_i \lambda \gamma n}{v_i+\lambda \gamma n}.
\end{align}
In particular, if we fix $\epsilon\leq P(w^0)-D(\alpha^0)$, then for
 \begin{equation}\label{eq:complexity-intro2} 
T\geq \max_i \left( \frac{1}{p_i} + \frac{v_i}{p_i \lambda \gamma n}\right)  \log \left( \frac{P(w^0)-D(\alpha^0)}{\epsilon} \right),\end{equation}
we are guaranteed that $\Exp[P(w^T)-D(\alpha^T)]\leq \epsilon.$
\end{theorem}


A result of a similar flavour but for a different problem and not in a primal-dual setting has been established in \cite{NSync}, where the authors analyze a parallel coordinate descent method, NSync, also  with an {\em arbitrary sampling}, for minimizing  a strongly convex function under an ESO assumption.

In the rest of this section we will specialize the above result to a few selected samplings. We then devote two separate sections to Quartz specialized to the $\tau$-nice sampling (Section~\ref{sec:MAIN-tau-nice}) and Quartz specialized to the $(c,\tau)$-distributed sampling (Section~\ref{MAIN:distributed} -- as we do a more detailed analysis of the results in these two cases. 

\subsection{Quartz with uniform serial sampling}\label{sec-uss}
We first look at the special case when $\hat{S}$ is the uniform serial sampling, i.e., when $p_i=1/n$ for all $i\in[n]$. 

\begin{corollary}\label{coro-univalue}
Assume that at each iteration of Quartz we update only one dual variable uniformly at random and 
use $v_i=\lambda_{\max} (A_i^\top A_i)$ for all $i\in[n]$. If we let $\epsilon\leq P(w^0)-D(\alpha^0)$ and
\begin{equation}\label{eq:s89s8hs}
T\geq  \displaystyle\left(n + \frac{ \max_i \lambda_{\max}(A_i^\top A_i)}{\lambda \gamma }\right)\log\left(\frac{P(w^0)-D(\alpha^0)}{ \epsilon}\right),
\end{equation}
then $\Exp[P(w^T)-D(\alpha^T)]\leq \epsilon$.
\end{corollary}
\begin{proof} The result follows by combining
Lemma~\ref{l-serialv} and Theorem~\ref{th-mainth}.
\end{proof}

 Corollary~\ref{coro-univalue} should  be compared with Theorem 5 
in \cite{SDCA} (covering the L2-regularized case) and Theorem 1 in \cite{ASDCA} (covering the case of general $g$).  They obtain the rate
$$
\left(n+ \frac{ \max_i \lambda_{\max}(A_i^\top A_i)}{\lambda \gamma }\right)\log\left(\left(n+\frac{ \max_i \lambda_{\max}(A_i^\top A_i)}{\lambda \gamma }\right)\left(\frac{D(\alpha^*)-D(\alpha^0)}{ \epsilon}\right)\right),
$$
where $\alpha^*$ is the dual optimal solution.  Notice that the dominant terms in the two rates exactly  match, although our logarithmic term is better and simpler.

\subsection{Quartz with optimal serial sampling (importance sampling)}

 According to Lemma~\ref{l-serialv}, the parameter $v$ for a serial sampling $\hat S$ is determined by~\eqref{a-vi} and is {\em independent} of the distribution of $\hat S$. We can then seek to maximize the quantity $\theta$ in~\eqref{a-theta} to obtain the best bound. A simple calculation reveals that the optimal probability is given by:
\begin{align}\label{a-optimalp}
\Prob(\hat{S}=\{i\}) = p^*_i \eqdef \frac{\lambda_{\max}(A_i^\top A_i)+\lambda \gamma n}{\sum_{i=1}^n \left(\lambda_{\max}(A_i^\top A_i)+\lambda \gamma n\right)}.
\end{align}
Using this sampling,  we obtain the following  iteration complexity bound, which is an improvement on the  bound for uniform probabilities \eqref{eq:s89s8hs}.
 
\begin{corollary}\label{coro-opimp}
Assume that at each iteration of Quartz we update only one dual variable at random according to the probability $p^*$ defined in~\eqref{a-optimalp} and 
use $v_i=\lambda_{\max} (A_i^\top A_i)$ for all $i\in[n]$. If we let $\epsilon\leq P(w^0)-D(\alpha^0)$  and
\begin{equation}\label{eq:is789ddj}
T\geq  \displaystyle\left(n + \frac{ \tfrac{1}{n}\sum_{i=1}^n\lambda_{\max}(A_i^\top A_i)}{\lambda \gamma}\right)\log\left(\frac{P(w^0)-D(\alpha^0)}{ \epsilon}\right),
\end{equation}
then  $\Exp[P(w^T)-D(\alpha^T)]\leq \epsilon$.
\end{corollary}

Note that in contrast with the serial uniform sampling, we now have dependence on the {\em average} of the eigenvalues. The above result should be compared with the complexity result of  Iprox-SDCA~\cite{IProx-SDCA}:
$$
\left( n+\frac{ \tfrac{1}{n}\sum_{i=1}^n\lambda_{\max}(A_i^\top A_i)}{\lambda \gamma }\right) \log\left( \left( n+ \frac{ \tfrac{1}{n} \sum_{i=1}^n \lambda_{\max}(A_i^\top A_i)}{\lambda \gamma }\right) \left(\frac{D(\alpha^*)-D(\alpha^0)}{ \epsilon}\right)\right),
$$
where $\alpha^*$ is the dual optimal solution.  Again, the dominant terms in the two rates exactly  match, although our logarithmic term is better and simpler.

\subsection{Quartz with product sampling}

In this section  we  apply Theorem~\ref{th-mainth} to the  case when $\hat S$ is the product sampling (see the description in Section~\ref{subsec-products}).  All the notation we use here was established there. 

\begin{corollary}\label{coro-ps}
Under the group separability assumption,  let $\hat S$ be the product sampling and  let $v_i=\lambda_{\max} (A_i^\top A_i)$ for all $i\in[n]$.  If we fix $\epsilon\leq P(w^0)-D(\alpha^0)$ and
\begin{equation*}
T\geq  \displaystyle \max_i\left(|X_{l_i}|  + \frac{  \lambda_{\max}(A_i^\top A_i) |X_{l_i}|}{\lambda \gamma n}\right)\log\left(\frac{P(w^0)-D(\alpha^0)}{ \epsilon}\right),
\end{equation*}
then $\Exp[P(w^T)-D(\alpha^T)]\leq \epsilon$.
\end{corollary}
\begin{proof}
The proof follows directly from Theorem~\ref{th-mainth}, Lemma~\ref{l-pips} and~\eqref{a-pips}.
\end{proof}
Recall from Section~\ref{subsec-products} that the product sampling $\hat S$ has cardinality $\tau \geq 1$ and
 is non-uniform as long as all the sets $\{X_1,\dots,X_{\tau}\}$ do not have the same cardinality.
To the best of our knowledge, Corollary~\ref{coro-ps} is the first \textit{explicit} complexity bound of stochastic algorithm 
using non-serial and non-uniform sampling for composite convex optimization problem (the paper~\cite{NSync} only deals with smooth functions and the method is not primal-dual), albeit
under the group separability assumption.

Let us compare the complexity bound with the serial uniform case (Corollary~\ref{coro-univalue}):
$$
\frac{n + \frac{ \max_i \lambda_{\max}(A_i^\top A_i)}{\lambda \gamma }}{\max_i\left(|X_{l_i}|  + \frac{  \lambda_{\max}(A_i^\top A_i) |X_{l_i}|}{\lambda \gamma n}\right)}
\geq \min_i \frac{n}{|X_{l_i}|}.
$$
Hence the iteration bound of Quartz specialized to product sampling is at most a $\max_{i} |X_{l_i}|/n$ fraction of 
that of Quartz specialized to serial uniform sampling. The factor $\max_{i} |X_{l_i}|/n$ varies from $1/\tau$ to $1$, depending 
on the degree to which  the partition $X_1,\dots,X_{\tau}$ is balanced. A perfect linear speedup ($\max_{i} |X_{l_i}|/n=1/\tau$)
only occurs when the partition $X_1,\dots,X_{\tau}$ is perfectly balanced (i.e., the set $X_l$ have the same cardinality), in which case the product sampling is uniform (recall the definition of uniformity we use in this paper: $\Prob(i\in \hat S)=\Prob(i'\in \hat S)$ for all $i,i'\in [n]$). Note that if the  partition is not perfectly but sufficiently so, then the factor $\max_{i} |X_{l_i}|/n$
will be  close to the perfect linear speedup factor $1/\tau$.

\section{Quartz with $\tau$-nice Sampling (standard mini-batching)} \label{sec:MAIN-tau-nice}

We now specialize   Theorem~\ref{th-mainth}  to the case of the $\tau$-nice sampling.

\begin{corollary}\label{coro-taunice2}
Assume $\hat S$ is the $\tau$-nice sampling and $v$ is chosen as in~\eqref{a-vtaunice}. If we let $\epsilon\leq P(w^0)-D(\alpha^0)$ and 
\begin{align}\label{a-Tgeqtaunice2}
T\geq  \displaystyle
\left(\frac{n}{\tau}  + \frac{ \max_i \lambda_{\max}\left(\sum_{j=1}^d \left(1+\frac{(\omega_j-1)(\tau-1)}{n-1}\right) A_{ji}^\top A_{ji}\right)}{\lambda \gamma \tau}\right)\log\left(\frac{P(w^0)-D(\alpha^0)}{ \epsilon}\right),
\end{align}
then  $\Exp[P(w^T)-D(\alpha^T)]\leq \epsilon.$
\end{corollary}
\begin{proof} The result follows by combining
Lemma~\ref{prop-ESO} and Theorem~\ref{th-mainth}.
\end{proof}

Let us now have a detailed look at the above result; especially in terms of how it compares with the serial uniform case (Corollary~\ref{coro-univalue}). We do this comparison in Table~\ref{tab-comtauxxx}. For fully sparse data, we get {\em perfect linear speedup}: the bound in the second line of Table~\ref{tab-comtauxxx} is a $1/\tau$ fraction of the bound in the first line. For fully dense data, the condition number ($\kappa\eqdef\max_i \lambda_{\text{max}}(A_i^\top A_i)/(\gamma\lambda)$) is unaffected by mini-batching/parallelization. Hence, linear speedup is obtained if $\kappa=O(n/\tau)$.  For general data, the behaviour of Quartz with $\tau$-nice sampling interpolates these two extreme cases. That is, $\kappa$ gets multiplied by a quantity between $1/\tau$ (fully sparse case) and $1$ (fully dense case). It is convenient to write this factor in the form
\[\frac{1}{\tau}\left(1+\frac{(\tilde{\omega}-1)(\tau-1)}{n-1}\right),\]
where $\tilde{\omega}\in [1,n]$ is a measure of {\em average sparsity} of the data, using which we can write:

\begin{eqnarray}\label{a-Tgeqtaunice2xxx}
T(\tau) &\eqdef& \left(\frac{n}{\tau}  + \frac{ \left(1+\frac{(\tilde{\omega}-1)(\tau-1)}{n-1}\right) \max_i \lambda_{\max}(A_i^\top A_i) }{\lambda \gamma \tau}\right) \log\left(\frac{P(w^0)-D(\alpha^0)}{ \epsilon}\right).
\end{eqnarray}

\begin{table}[!ht]
\centering
 \begin{tabular}{| >{\centering\arraybackslash}m{0.15\textwidth}  | >{\centering\arraybackslash}m{1.2in} | >{\centering\arraybackslash}m{2.7in} | >{\centering\arraybackslash}m{0.8in}|}
\hline
Sampling $\hat{S}$ & Data & Complexity of Quartz \eqref{a-Tgeqtaunice2} & Theorem \\
\hline
Serial uniform  & Any data  & $$ n +\frac{\max_i \lambda_{\text{max}}(A_i^\top A_i)}{\lambda \gamma}$$ & Corollary \ref{coro-univalue}  \\
\hline
$\tau$-nice  & Fully sparse data ($\omega_j=1$ for all $j$)  & $$ \frac{n}{\tau}+\frac{\max_i \lambda_{\text{max}}(A_i^\top A_i)}{\lambda \gamma \tau}$$ &  Corollary~\ref{coro-taunice2}\\
\hline
$\tau$-nice & Fully dense data ($\omega_j=n$ for all $j$) & $$ \frac{n}{\tau}+\frac{\max_i \lambda_{\text{max}}(A_i^\top A_i)}{\lambda \gamma}$$ &  Corollary~\ref{coro-taunice2}\\
\hline
$\tau$-nice  & Any data  & $$ \frac{n}{\tau}+\frac{\left(1+\frac{(\tilde{\omega}-1)(\tau-1)}{n-1}\right)\max_i \lambda_{\text{max}}(A_i^\top A_i)}{\lambda \gamma \tau}$$ &  Corollary~\ref{coro-taunice2}\\
\hline
 \end{tabular}
\caption{Comparison of the complexity of Quartz with serial uniform sampling and $\tau$-nice sampling. }\label{tab-comtauxxx}
\end{table}

\subsection{Theoretical speedup factor}\label{sec:tau-nice-speedup-factor}

For simplicity of exposition,  let us now assume that $\lambda_{\text{max}}(A_i^\top A_i)=1$. We will now study the {\em theoretical speedup factor}, defined as:
\begin{align}\label{a-spth}
\frac{T(1)}{T(\tau)}
\overset{\eqref{a-Tgeqtaunice2xxx}}{=}\frac{\tau(1+\lambda\gamma n)}{1+\lambda \gamma n+\frac{(\tau-1)(\tilde{\omega}-1)}{(n-1)}}
=\frac{\tau}{1+\frac{(\tau-1)(\tilde{\omega}-1)}{(n-1)(1+\lambda\gamma n)}}\enspace.
\end{align}

That is, the speedup factor measures how much better Quartz is with $\tau$-nice sampling than in the serial uniform case (with $1$-nice sampling). Note that the speedup factor  is a concave and increasing function
 with respect to the number of threads $\tau$.  The value depends on two factors: the {\em relative sparsity} level of the data matrix $A$, expressed through the quantity
$\frac{\tilde{\omega}-1}{n-1}$ and the condition number of the problem, expressed through the quantity $\lambda\gamma n$. 
We provide below two lower bounds for the speedup factor:
\begin{align}\label{a-Twn}
\frac{T(1,1)}{T(1,\tau)} \geq \frac{\tau}{1+\frac{\tilde{\omega}-1}{n-1}}\geq  \frac{\tau}{2}\qquad \text{if} \qquad  1\leq  \tau\leq 2+\lambda \gamma n\enspace.
\end{align}

Note that the last term does not involve $\tilde{\omega}$. In other words,  linear speedup (modulus a factor of 2) is achieved at least until
$\tau=2+\lambda\gamma n$ (of course, we also require that $\tau\leq n$), regardless of the data matrix $A$. For instance, if $\lambda\gamma=1/\sqrt{n}$, which is a frequently used setting for the regularizer, then we get  {\em data independent linear speedup} up to mini-batch size $\tau=2+ \sqrt{n}$. Moreover, from the first inequality in \eqref{a-Twn} we see that there is {\em further data-dependent speedup}, depending on the average sparsity measure $\tilde{\omega}$. We give an illustration of this phenomenon in Figure~\ref{fig:thspeeduptau}, where
we plot the theoretical speedup factor~\eqref{a-spth} as a function of the number of threads $\tau$, for $n=10^6$, $\gamma=1$ and three values of $\tilde{\omega}$ and $\lambda$. Lookingat the plots from right to left, we see that for fixed $\lambda$, the speedup factor increases as  $\tilde{\omega}$ decreases, as described by~\eqref{a-spth}.  Moreover, as the regularization parameter $\lambda$ gets smaller an reaches the value $1/n$, the speedup factor is healthy for sparse data only.  However, for $\lambda=1/\sqrt{n}=10^{-3}$, we observe linear speedup up to   $\tau=\sqrt{n}=1000$, regardless of $\tilde{\omega}$ (the sparsity of the data), as predicted. There is additional data-driven speedup beyond this point, which is better for sparser data.

\begin{figure}[ht]
\centering
\subfigure[$\tilde{\omega} =10^2$, $n=10^6$, $\gamma=1$]{
     \includegraphics[width=0.3\textwidth]{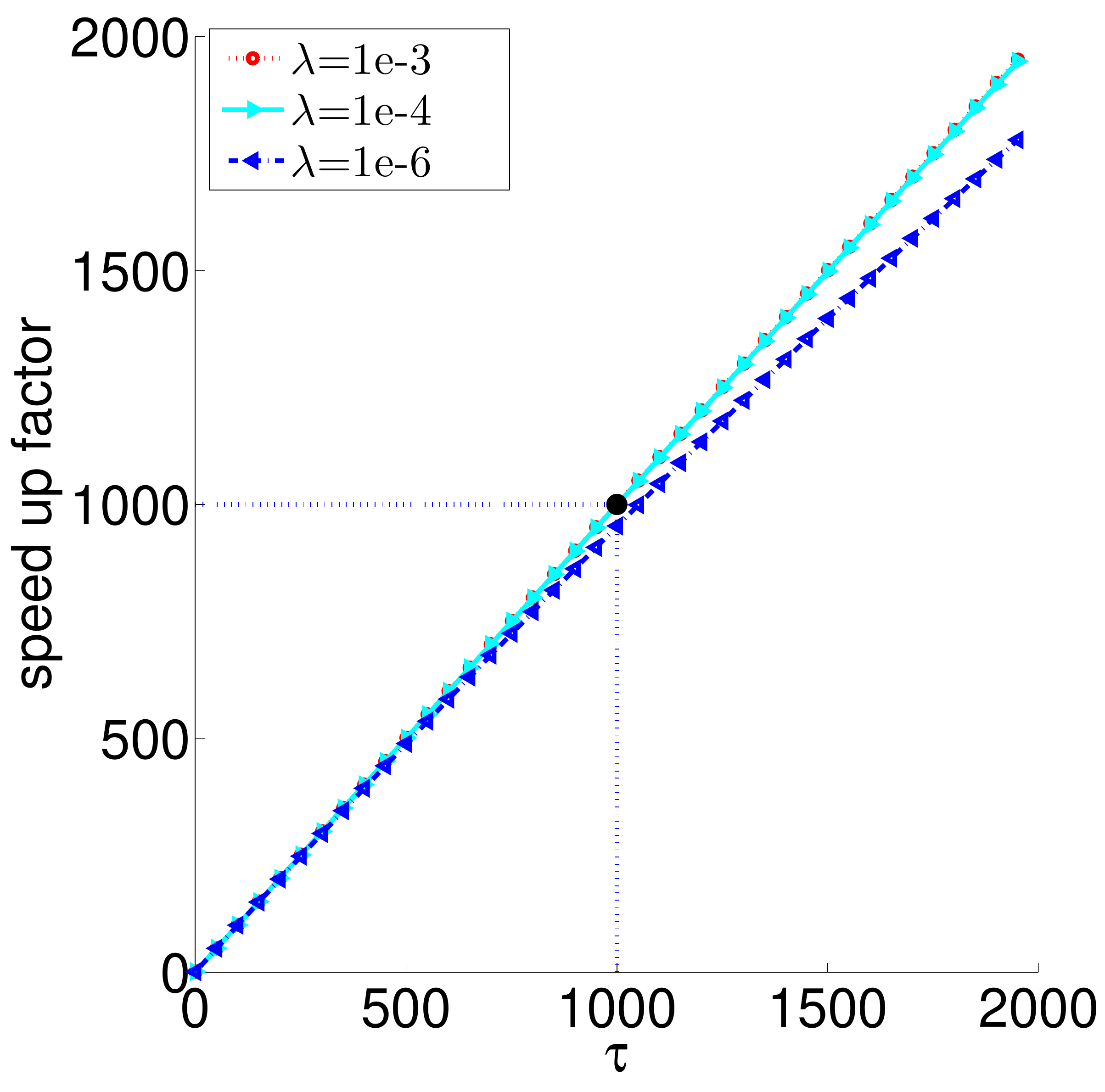}
    \label{figs:subfig1}
}
\subfigure[$\tilde{\omega}=10^4$, $n=10^6$, $\gamma=1$]{
    \includegraphics[width=0.3\textwidth]{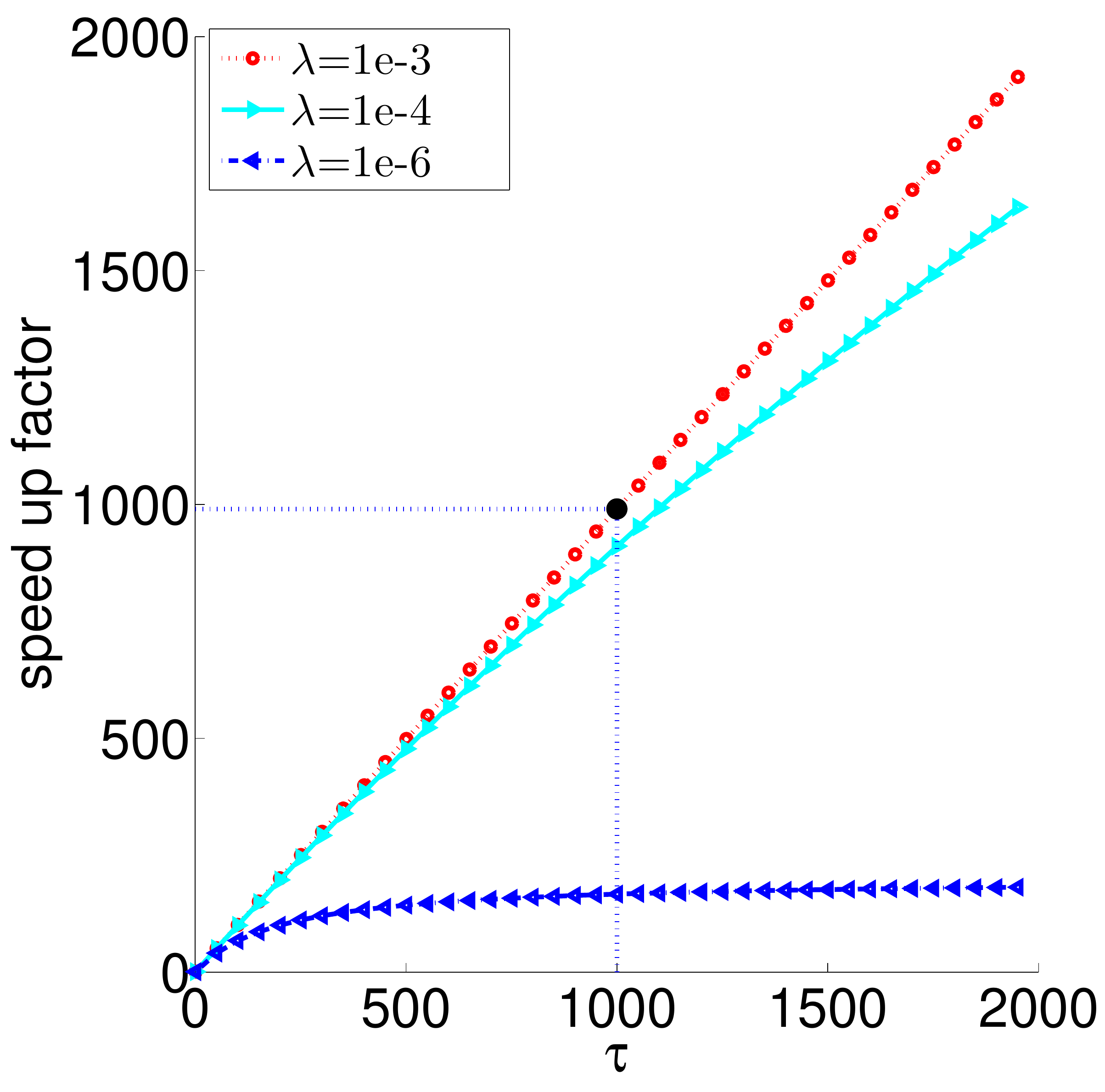}
    \label{figs:subfig2}
}
\subfigure[$\tilde{\omega}=10^6$, $n=10^6$, $\gamma=1$]{
    \includegraphics[width=0.3\textwidth]{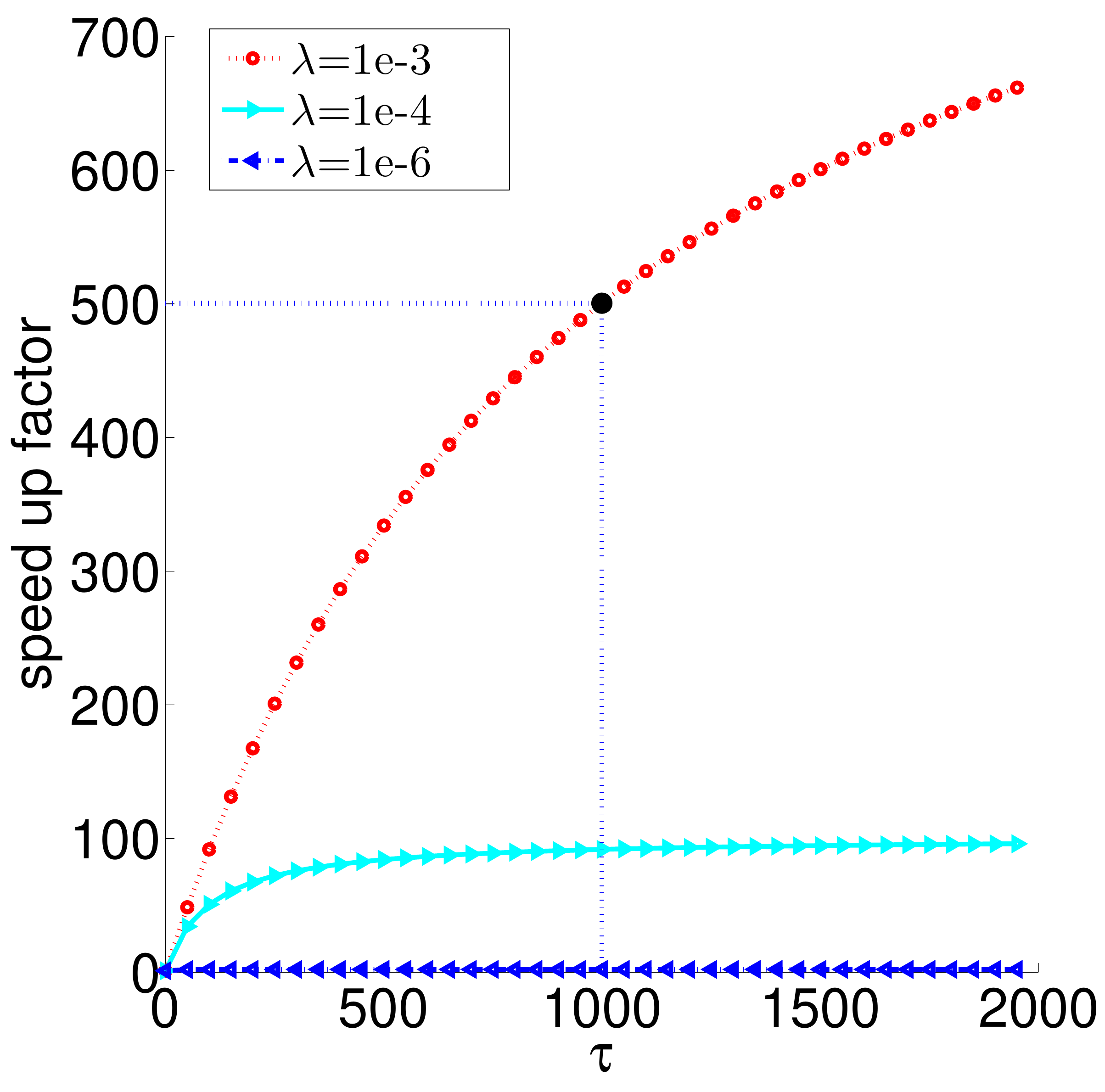}
    \label{figs:subfig3}
}

\caption[Optional caption for list of figures]{The speedup factor~\eqref{a-spth} as a function of $\tau$ for $n=10^6$,  $\gamma=1$, three regularization parameters and data of various sparsity levels.}
\label{fig:thspeeduptau}
\end{figure}

\subsection{Quartz vs existing primal-dual mini-batch methods} 

We now compare the above result with existing mini-batch stochastic dual coordinate ascent methods. A mini-batch variant of SDCA, to which Quartz with $\tau$-nice sampling can be naturally compared,   has been proposed and analyzed  previously in~\cite{pegasos2},~\cite{ASDCA} and~\cite{SPDC}. In~\cite{pegasos2},  the authors proposed to use a so-called \textit{safe mini-batching}, 
which is precisely equivalent  to finding the stepsize parameter $v$ satisfying Assumption~\ref{ass-ESO} (in the special case of $\tau$-nice sampling). However, they only analyzed the case where the functions $\phi_i$ are non-smooth.  In~\cite{ASDCA}, the authors studied  accelerated mini-batch SDCA (ASDCA), specialized to the case when the regularizer $g$ is the squared L2 norm. They showed that the complexity of ASDCA interpolates between that of SDCA and  accelerated gradient descent (AGD)~\cite{Nesterov07} through varying the mini-batch size $\tau$. In~\cite{SPDC}, the authors proposed a mini-batch extension of
 their stochastic primal-dual coordinate algorithm (SPDC). Both ASDCA and SPDC
 reach the same complexity as AGD when the 
mini-batch size equals to $n$, thus should be considered as accelerated algorithms.  The complexity bounds for all these algorithms are summarized in Table~\ref{tab-comtau}. To facilitate the comparison, we assume that
$ \max_i \lambda_{\max}(A_i^\top A_i)=1$ (since the analysis of ASDCA assumes this). In Table~\ref{tab-comtau987987} we  compare the complexities of SDCA, ASDCA, SPDC and Quartz in several regimes. We have used Lemma~\ref{lem:inequalities} to simplify  the bounds for Quartz. 

\begin{table}[!ht]
\centering
 \begin{tabular}{| >{\centering\arraybackslash}m{0.2\textwidth} | >{\centering\arraybackslash}m{3in} | >{\centering\arraybackslash}m{0.5in}|}
\hline
Algorithm & Iteration complexity & $g$\\
\hline
SDCA \cite{SDCA}  & $n + \tfrac{1}{\lambda \gamma}$ & $\tfrac{1}{2}\|\cdot\|^2$ \\
\hline
ASDCA \cite{ASDCA} & $4 \times  \max\left\{\frac{n}{\tau},\sqrt{\frac{n}{\lambda\gamma \tau}},\frac{1}{\lambda \gamma \tau},\frac{n^{\frac{1}{3}}}{(\lambda \gamma \tau)^{\frac{2}{3}}}\right\}$ & $\tfrac{1}{2}\|\cdot\|^2$\\
\hline 
SPDC \cite{SPDC} & $\frac{n}{\tau}+\sqrt{\frac{n}{\lambda \gamma \tau}}$ & general \\ \hline 
\bf{Quartz with $\tau$-nice sampling} & $\frac{n}{\tau}+\left(1+\frac{(\tilde \omega -1)(\tau-1)}{n-1}\right)\frac{1}{\lambda \gamma \tau}$ &  general\\
\hline
 \end{tabular}
\caption{Comparison of the iteration complexity of several primal-dual algorithms performing  stochastic coordinate ascent steps in the dual using a mini-batch of examples of size $\tau$ (with the exception of SDCA, which is a serial method using $\tau=1$. We assume that $\lambda_{\max}(A_i^\top A_i) = 1$ for all $i$ to facilitate comparison since this assumption has been implictly made in \cite{ASDCA}.}\label{tab-comtau}
\end{table}

\begin{table}[ht]
\centering

 \begin{tabular}{| >{\centering\arraybackslash}m{0.10\textwidth} | >{\centering\arraybackslash}m{0.95in} | >{\centering\arraybackslash}m{0.95in} | >{\centering\arraybackslash}m{0.95in}| >{\centering\arraybackslash}m{0.95in}| >{\centering\arraybackslash}m{0.95in}|}
\hline
Algorithm & $\gamma \lambda n = \Theta(\tfrac{1}{\sqrt{n}})$ & $\gamma\lambda n = \Theta(\tfrac{1}{\tau})$ & $\gamma\lambda n = \Theta(1)$ & $\gamma\lambda n = \Theta(\tau)$ & $\gamma\lambda n = \Theta(\sqrt{n})$\\
\hline
   & $\kappa = n^{3/2}$ & $\kappa = n\tau$ & $\kappa = n$ & $\kappa = n/\tau$ & $\kappa = \sqrt{n}$\\
\hline
\hline
SDCA \cite{SDCA} & $n^{3/2}$ & $n \tau$ & $n$ & $n$ & $n$\\
\hline 
ASDCA \cite{ASDCA} &  $n^{3/2}/\tau + n^{5/4}/\sqrt{\tau} + n^{4/3}/\tau^{2/3}$ & $n$ & $n/\sqrt{\tau}$ & $n/\tau$ &  $n/\tau + n^{3/4}/\sqrt{\tau}$\\
\hline 
SPDC \cite{SPDC} & $n^{5/4}/\sqrt{\tau}$ & $n$ & $n/\sqrt{\tau}$ & $n/\tau$ & $n/\tau + n^{3/4}/\sqrt{\tau}$\\ 
\hline 
\bf{Quartz ($\tau$-nice)} & $n^{3/2}/\tau + \tilde{\omega}\sqrt{n}$ & $n + \tilde{\omega}\tau$ & $n/\tau +\tilde{\omega}$ & $n/\tau$ & $n/\tau + \tilde{\omega}/\sqrt{n}$\\
\hline
 \end{tabular}
\caption{Comparison of leading factors in the complexity bounds of several methods in 5 regimes; where $\kappa=1/(\gamma\lambda)$ is the condition number. We ignore constant terms and hence one can replace each ``plus'' by a ``max''.}\label{tab-comtau987987}
\end{table}

\begin{lemma} \label{lem:inequalities} For any $\tilde{\omega}\in [1,n]$ and $\tau\in [1,n]$ we have 
\[\frac{(\tilde{\omega}-1)(\tau-1)}{n-1} \leq \frac{\tilde{\omega} \tau}{n} \leq 1+\frac{(\tilde{\omega}-1)(\tau-1)}{n-1} \leq 1+ \frac{\tilde{\omega} \tau}{n}.\]
\end{lemma}
\begin{proof} The second inequality follows by showing that the function $\phi_1(x) = x + \tfrac{(\tilde{\omega}-x)(\tau-x)}{n-x}$ is increasing, the first and third follow by showing that $\phi_2(x) = \tfrac{(\tilde{\omega}-x)(\tau-x)}{n-x}$ is decreasing on $[0,1]$. The monotonicity claims follow from the fact that
$\phi_1'(x) = \tfrac{n^2 + \tilde{\omega} \tau - n(\tilde{\omega} + \tau)}{(n-x)^2} = \tfrac{(n-\tilde{\omega})(n-\tau)}{(n-x)^2} \geq 0$
and
$\phi_2'(x) = \phi_1'(x) -1 = \tfrac{(n-\tilde{\omega})(n-\tau) - (n-x)^2}{(n-x)^2} \leq 0$ for all $x\in [0,1]$.
\end{proof}

Looking at Table~\ref{tab-comtau987987}, we see that in the $\gamma\lambda n=\Theta(\tau)$ regime (i.e., if the condition number is $\kappa=\Theta(n/\tau)$),  
Quartz matches the linear speedup (when compared to SDCA) of ASDCA and SPDC. When the condition number is roughly equal to the sample size ($\kappa=\Theta(n)$), 
then Quartz does better than both ASDCA and SPDC as long as $n/\tau + \tilde{\omega} \leq n/\sqrt{\tau}$. In particular, 
this is the case when the data is sparse: $\tilde{\omega}\leq n/\sqrt{\tau}$. If the data is even more sparse (and in many big data applications one has $\tilde{\omega}=O(1)$) and we have $\tilde{\omega}\leq n/\tau$, then Quartz significantly outperforms both ASDCA and SPDC. Note that Quartz can be better than both  ASDCA and SPDC even in the domain of accelerated methods, that is, when the condition number is larger than the number of examples: \begin{equation}\label{eq:siu98shs}\kappa =\frac{1}{\gamma \lambda} \geq n.\end{equation} Indeed, we have the following result, which can be interpreted as follows:  if $\kappa \leq \tau n/4$ (that is, $\lambda \gamma \tau n\geq 4$), then there are sparse-enough problems for which  Quartz is better than both ASDCA and SPDC.

\begin{proposition}
Assume that \eqref{eq:siu98shs} holds and that $\max_i \lambda_{\text{max}} (A_i^\top A_i) =1$. Then if the data is sufficiently sparse so that
\begin{align}\label{a-kgtng}
\lambda \gamma \tau n \geq \left(2+\frac{(\tilde \omega-1)(\tau-1)}{n-1}\right)^2,
\end{align}
 the  iteration complexity (in $\tilde O$ order) of Quartz is better than that of ASDCA and SPDC.
\end{proposition}
\begin{proof}
As long as $\lambda \gamma \tau n\geq 1$, which holds under our assumption, the iteration complexity of ASDCA is:
$$\tilde O \left( \max\left\{\frac{n}{\tau},\sqrt{\frac{n}{\lambda\gamma \tau}},\frac{1}{\lambda \gamma \tau},\frac{n^{\frac{1}{3}}}{(\lambda \gamma \tau)^{\frac{2}{3}}}\right\}\right)=\tilde O\left(\sqrt{\frac{n}{\lambda \gamma \tau}}\right).$$
which is already less than that of SPDC. Moreover, 
$$
\sqrt{\frac{n}{\lambda \gamma \tau}}\overset{\eqref{a-kgtng}}{\geq} \frac{2+\frac{(\tau-1)(\tilde \omega -1)}{n-1}}{\lambda \gamma \tau}
\overset{\eqref{eq:siu98shs}}{\geq} \frac{n}{\tau}+\frac{1+\frac{(\tau-1)(\tilde \omega -1)}{n-1}}{\lambda \gamma \tau}.
$$
\end{proof}


\section{Quartz with Distributed Sampling} \label{MAIN:distributed}

In this section we  apply Theorem~\ref{th-mainth} to the  case when $\hat S$ is the $(c,\tau)$-distributed sampling; see the description of this sampling in Section~\ref{sec:ESO-distributed}.

\begin{corollary}\label{coro-distri}
Assume that $\hat{S}$ is a  $(c,\tau)$-distributed sampling and $v$ is chosen as in~\eqref{a-bDiibetaj}.  If we let $\epsilon \leq P(w^0)-D(\alpha^0)$ and 
\begin{align}\label{a-Tgeqtaunice}
T\geq T(c,\tau) \times \log\left(\frac{P(w^0)-D(\alpha^0)}{ \epsilon}\right),
\end{align}
where 
\begin{align}\label{a-Tctau}
T(c,\tau)  \eqdef  \frac{n}{c\tau} + \max_i\frac{\lambda_{\max}\left( \sum_{j=1}^d \left(1+\frac{(\tau-1)(\omega_j-1)}{\max\{n/c-1,1\}}+  \left(\frac{\tau c}{n} - \frac{\tau-1}{\max\{n/c-1,1\}}\right) \frac{\omega_j'-1}{\omega_j'}\omega_j\right) A_{ji}^\top A_{ji}\right)}{\lambda\gamma c\tau},
\end{align}
then $\Exp[P(w^T)-D(\alpha^T)]\leq \epsilon.$
\end{corollary}
\begin{proof}
If $\hat S$ is a $(c,\tau)$-distributed sampling, then 
$$p_i=\frac{c\tau}{n},\qquad i\in[n].$$
It now only remains to combine Theorem~\ref{th-mainth} and Lemma~\ref{prop-distriESO}.
\end{proof}


The expression~\eqref{a-Tctau} involves $\omega_j'$, which depends on the partitioning 
 $\{{\cal P}_1$, ${\cal P}_2$, \dots, ${\cal P}_c\}$ of the dual variable and the data. The following lemma says that
the effect of the partition is negligible, and in fact vanishes as $\tau$ increases.  It was proved  in~\cite[Lemma 5.2]{Hydra2}.

\begin{lemma}[\cite{Hydra2}]\label{l-Hy2beta1beta2}
If $n/c\geq 2$ and $\tau\geq 2$, then for all $j\in [d]$, we have
 $$\left(\frac{\tau c}{n} - \frac{\tau-1}{n/c-1}\right) \frac{\omega_j'-1}{\omega_j'}\omega_j\leq
 \frac{1}{\tau-1}\left(1+\frac{(\tau-1)(\omega_j-1)}{n/c-1}\right).$$
\end{lemma}
According to this result, when each node owns at least two dual examples ($n/c\geq 2$) and picks and updates at least two examples in each iteration ($\tau\geq 2$), then
\begin{eqnarray}
T(c,\tau) &\leq&  \frac{n}{c\tau} + \left(1+\frac{1}{\tau-1}\right) \frac{\max_i\lambda_{\max}\left( \sum_{j=1}^d \left(1+\frac{(\tau-1)(\omega_j-1)}{n/c-1}\right) A_{ji}^\top A_{ji}\right)}{\lambda\gamma c\tau}\notag \\
&=& \frac{n}{c\tau} + \left(1+\frac{1}{\tau-1}\right)\left(1+\frac{(\tau-1)(\hat{ \omega}-1)}{n/c-1}\right) \frac{\max_i \lambda_{\text{max}}(A_i^\top A_i)}{\lambda \gamma c\tau}, \label{eq:s98s98hsdjj}
\end{eqnarray}
where $\hat{\omega}\in [1,n]$ is an {\em average sparsity measure} similar to that one we introduced in the study of $\tau$-nice sampling.  This bound is similar to that we obtained for the $\tau$-nice sampling; and can be interpreted in an analogous way. Note that as the first term ($n$) receives perfect mini-batch scaling (it is divided by $c\tau$), while the condition number $\max_i \lambda_{\text{max}}(A_i^\top A_i)/(\lambda \gamma)$ is divided by $c\tau$ but also multiplied by  $\left(1+\frac{1}{\tau-1}\right)\left(1+\frac{(\tau-1)(\hat{ \omega}-1)}{n/c-1}\right)$. However, this term is bounded by $2\hat{\omega}$, and hence if $\hat{\omega}$ is small, the condition number also receives a nearly perfect mini-batch scaling.

\subsection{Quartz vs DiSDCA} A distributed variant of SDCA, named DisDCA, has been proposed in~\cite{NIPSdistributedSDCA} and analyzed in~\cite{AnaDisSDCA13}. 
The authors of~\cite{NIPSdistributedSDCA} proposed a basic DisDCA variant (which was analyzed) and a practical  DisDCA variant (which was not analyzed). The complexity  of basic DisDCA was shown to be:
\begin{align}\label{a-TdisDCA}
\displaystyle\left(\frac{n}{c\tau} + \frac{\max_{i}\lambda_{\max}(A_i^\top A_i)}{\lambda \gamma }\right)\log\left(\frac{n}{c\tau} + \left(\frac{ \max_i\lambda_{\max}(A_i^\top A_i)}{\lambda \gamma }\right)\cdot\frac{D(\alpha^*)-D(\alpha^0)}{ \epsilon}\right),
\end{align}
where $\alpha^*$ is an optimal dual solution. Note that this rate is much worse than our rate. Ignoring the logarithmic terms, while the first expression $n/(c\tau)$ is the same in both results, 
if we replace  all $\omega_j$ by the upper bound $n$ and all $\omega_j'$ by the upper bound $c$ in~\eqref{a-Tctau}, then 
\begin{align*}
T(c,\tau)& \leq \frac{n}{c\tau} +  \left(\max_i \lambda_{\max}(A_i^\top A_i)\right)\cdot\frac{1+\frac{(\tau-1)(n-1)}{\max(n/c-1)}+(\frac{\tau c}{n} - \frac{\tau-1}{\max(n/c-1,1)}) \frac{c-1}{c}n}{\lambda \gamma c\tau}
\\&\leq \frac{n}{c\tau} + \frac{\max_i\lambda_{\max}(A_i^\top A_i)}{\lambda \gamma}.
\end{align*}
Therefore, the dominant term in~\eqref{a-Tgeqtaunice} is a strict lower bound of that in~\eqref{a-TdisDCA}. Moreover, 
it is clear that the gap between~\eqref{a-Tgeqtaunice} and~\eqref{a-TdisDCA} is large when the data is sparse. For instance, in the perfectly sparse case with $\hat{\omega}=1$, 
the bound \eqref{eq:s98s98hsdjj} for Quartz becomes
\[\frac{n}{c\tau} + \left(1+\frac{1}{\tau-1}\right) \frac{\max_i \lambda_{\text{max}}(A_i^\top A_i)}{\lambda \gamma c\tau},\]
which is much better than \eqref{a-TdisDCA}.



\subsection{Theoretical speedup factor}  \label{sec:distr-speedup}

%

In analogy with the discussion in Section~\ref{sec:tau-nice-speedup-factor}, we shall now analyze the theoretical speedup factor $T(1,1)/T(c,\tau)$ measuring the multiplicative amount by which Quartz specialized to the $(c,\tau)$-distributed sampling is better than Quartz specialized to the serial uniform sampling.



In Section~\ref{sec:MAIN-tau-nice}, we have seen how the speedup factor increases with $\tau$ when a mini-batch of examples is used  at each iteration following the $\tau$-nice sampling. As we have discussed before, this sampling is not particularly suitable for a distributed implementation (unless $\tau=n$; which in the big data setting where $n$ is very large may be asking for many more cores/threads that are available). This is because the implementation of updates using this sampling would either result in frequently idle nodes, or in increased data transfer. 

Often the data matrix $A$ is too large to be stored on a single node, or limited number of threads/cores are available per node.
We then want to implement Quartz in a distributed way ($c>1$). It is therefore necessary to understand how the speedup factor compares to the hypothetical situation in which we would have a large machine where all data could be stored (we ignore communication costs here) and hence a  $c\tau$-nice sampling could be implemented. That is, we are interested in comparing $T(c,\tau)$ (distributed implementation) and $T(1,c\tau)$ (hypothetical computer). If for simplicity of exposition we  assume that $\lambda_{\text{max}}(A_i^\top A_i)=1$, it is possible to argue that if $c\tau\leq n$, then \begin{equation}\label{eq:shs98shspp}\frac{T(1,1)}{T(c,\tau)} \approx \frac{T(1,1)}{T(1,c\tau)}.\end{equation}

In Figure~\ref{figcc:thspeeduptau} we plot the contour lines of the theoretical speedup factor in a log-log plot with axes corresponding to $\tau$ and $c$. The contours  are nearly perfect straight lines, which means that the speedup factor is approximately constant for those pairs $(c,\tau)$ for which $c\tau $ is the same. In particular, this means that \eqref{eq:shs98shspp} holds. Note that better speedup is obtained for sparse data then for dense data. However, in all plots we have chosen $\gamma=1$ and $\lambda=1/\sqrt{n}$; and hence we expect data independent linear speedup up to $c\tau =\Theta(\sqrt{n})$ -- a special line is depicted in all three plots which defines this contour.

\begin{figure}[!h]
\centering
\subfigure[$\omega=10^2$, $n=10^6$]{
     \includegraphics[width=0.3\textwidth]{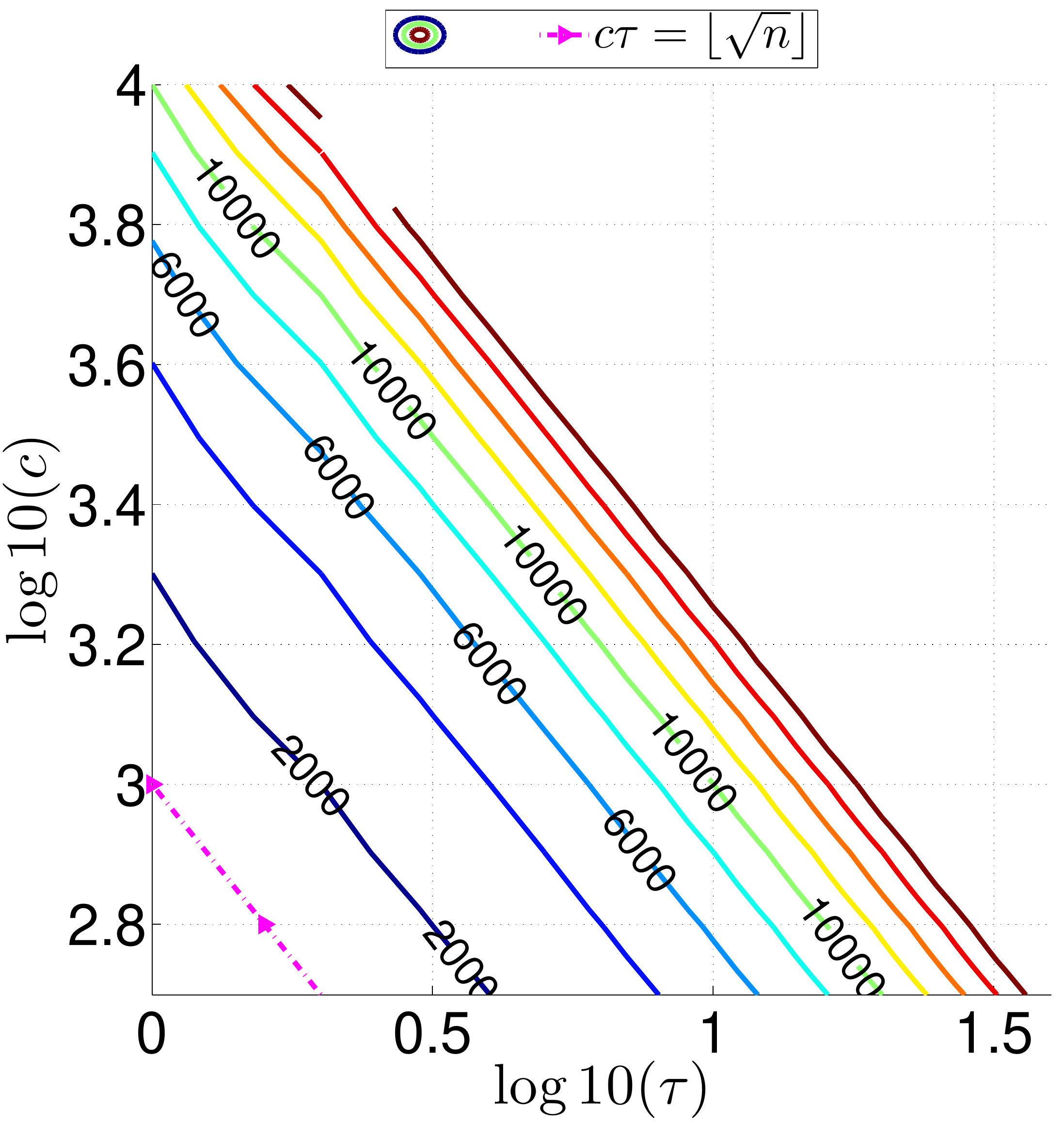}
    \label{figscc:subfig1}
}
\subfigure[$\omega=10^4$, $n=10^6$]{
    \includegraphics[width=0.3\textwidth]{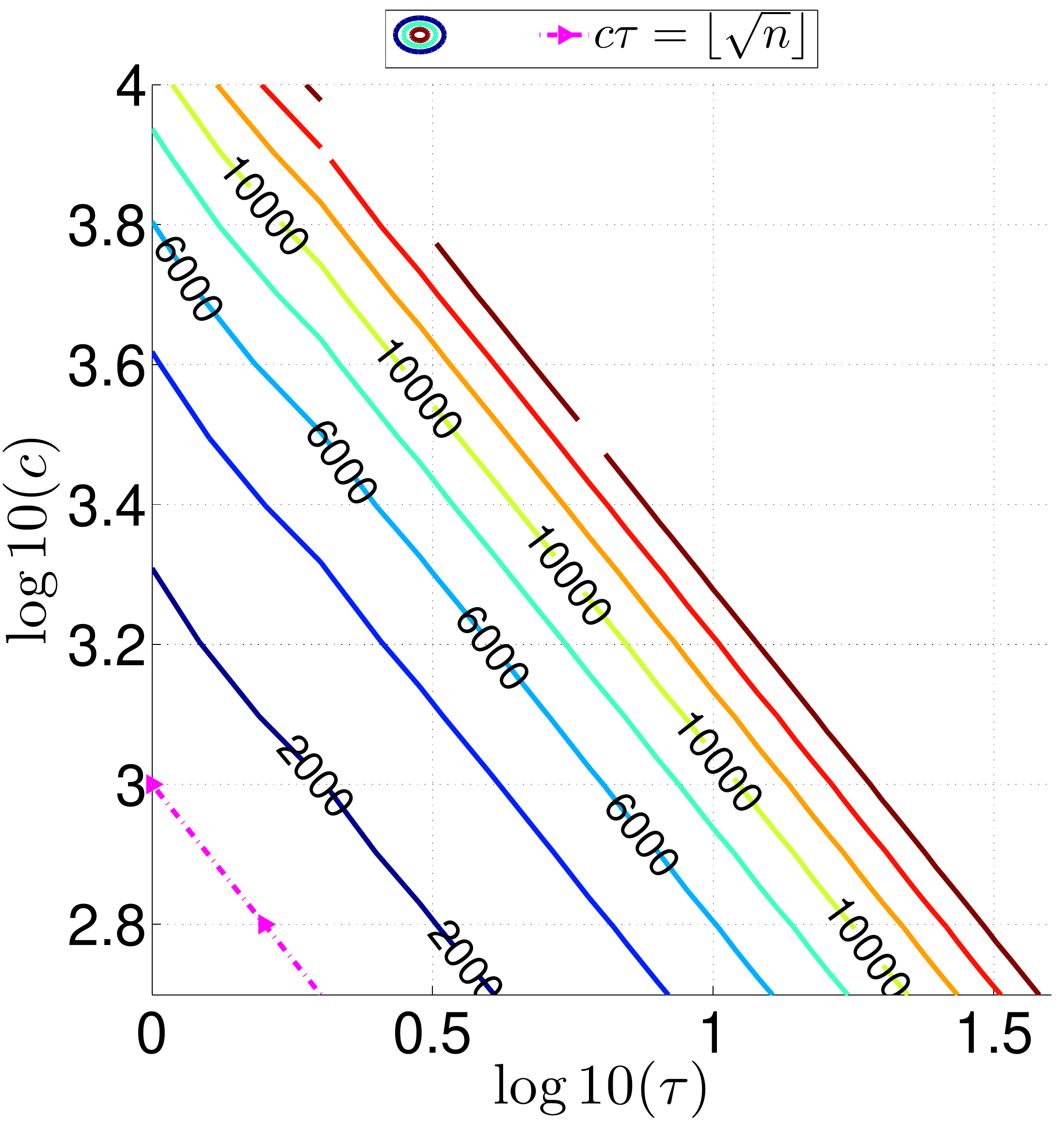}
    \label{figscc:subfig2}
}
\subfigure[$\omega=10^6$, $n=10^6$]{
    \includegraphics[width=0.3\textwidth]{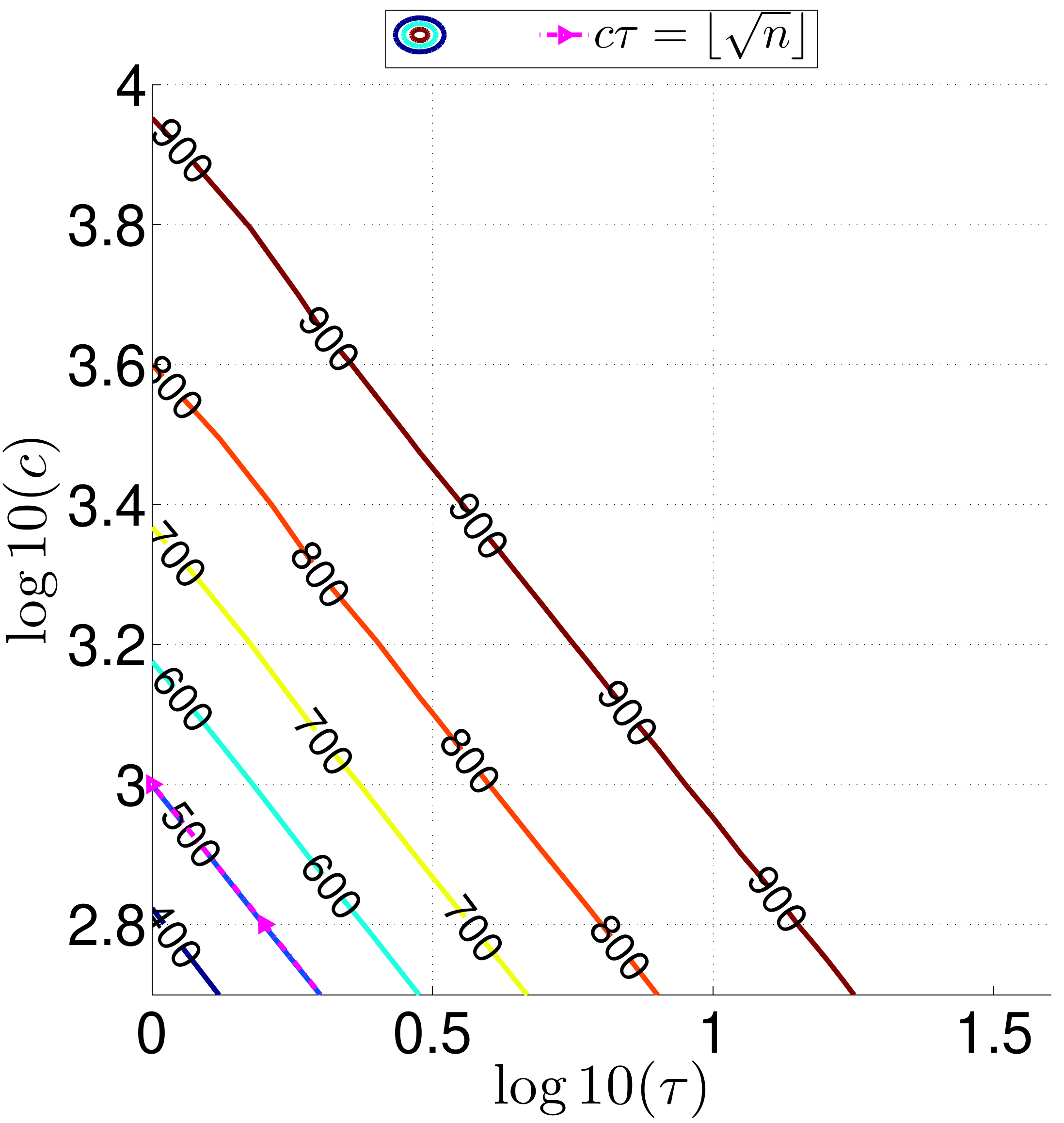}
    \label{figscc:subfig3}
}

\caption[Optional caption for list of figures]{Contour line plots of the speedup factor $T(1,1)/T(c,\tau)$ for $n=10^6$, $\gamma=1$, $\lambda=10^{-3}$, $\omega=10^2$ (Figure~\ref{figscc:subfig1}), 
$\omega=10^4$ (Figure~\ref{figscc:subfig2}), $\omega=10^6$ (Figure~\ref{figscc:subfig3}). Here, $\omega\in [1,n]$ is a degree of average sparsity of the data.}
\label{figcc:thspeeduptau}
\end{figure}

%
%
%

\section{Proof of the Main Result} \label{sec:proofs}

In this section we prove our main result (Theorem~\ref{th-mainth}). In order to make the analysis more transparent, we  will first  establish three auxiliary results.

\subsection{Three lemmas}

\begin{lemma}\label{lem:989sns}
Function $f:\R^{N}\rightarrow \R$ defined in~\eqref{a-defoff} satisfies the following inequality:
\begin{align}\label{a-LAA}
  f(\alpha+h)\leq  f( \alpha)+\<\nabla  f(\alpha),h>+\frac{1}{2\lambda n^2}h^\top A^\top A h,\qquad\forall \alpha,h\in \R^{N}
.\end{align}
\end{lemma}
\begin{proof}
Since $g$ is 1-strongly convex, $g^*$ is 1-smooth. Pick $\alpha,h\in \R^N$. Since, 
$f(\alpha) =\lambda g^*(\tfrac{1}{\lambda n} A \alpha)$,
we have 
\begin{eqnarray*}
f(\alpha+h) =  \lambda g^*\left(\tfrac{1}{\lambda n} A\alpha + \tfrac{1}{\lambda n} Ah\right) &\leq& \lambda \left(g^*\left(\tfrac{1}{\lambda n} A\alpha \right) + \ve{\nabla g^*\left(\tfrac{1}{\lambda n} A\alpha \right)}{\tfrac{1}{\lambda n} Ah} + \tfrac{1}{2}\left\|\tfrac{1}{\lambda n} Ah\right\|^2\right)\\
& = &  f(\alpha) +  \ve{\nabla f(\alpha)}{h} + \tfrac{1}{2\lambda n^2} h^T A^\top A h.
\end{eqnarray*}

%
%
%

\end{proof}

For $s=(s_1,\dots, s_n)\in \R^N$,  $h=(h_1,\dots,h_n)\in \
R^N$, where $s_i,h_i\in \R^m$ for all $i$,  we will for convenience write
\[\ve{s}{h}_p = \sum_{i=1}^n p_i \ve{s_i}{ h_i},\] 
where $p=(p_1,\dots,p_n)$ and  $p_i=\Prob(i\in \hat{S})$ for $i\in [n]$.

In the next lemma we give an expected separable overapproximation of the convex function $-D$.

\begin{lemma}\label{l-Dalphff}
If $\hat S$ and $v\in \R^n$ satisfy Assumption~\ref{ass-ESO}, then for all $\alpha,h\in \R^{N}$, the following holds:
\begin{equation}\label{a-Dalphah}
\begin{split}
&\Exp[-D(\alpha+h_{[\hat S]})]\\&\leq f(\alpha)+\<\nabla f(\alpha),h>_p+\frac{1}{2\lambda n^2}\|h\|_{p\cdot v}^2+\frac{1}{n}\sum_{i=1}^n\left[(1-p_i)\phi_i^*(-\alpha_i)+p_i\phi_i^*(-\alpha_i-h_i)\right].\end{split}
\end{equation}
\end{lemma}
\begin{proof}
By definition of $D$, we have
$$
-D(\alpha+h_{[\hat S]}) \overset{\eqref{eq:dual}}{=} f(\alpha+h_{[\hat S]})+\psi(\alpha+h_{[\hat S]}),
$$
where $f$ and $\psi$ are defined in~\eqref{a-defoff} and~\eqref{a-defofpsi}. Now we apply Lemma~\ref{lem:989sns} and~\eqref{a-ESO} to bound the first term:
\begin{align*}
\Exp[f(\alpha+h_{[\hat S]})]& \overset{\eqref{a-LAA}}{\leq} 
\Exp[f( \alpha)+\<\nabla  f(\alpha),h_{[\hat S]}>+\frac{1}{2\lambda n^2}h_{[\hat S]}^\top A^\top A h_{[\hat S]}]\\
&\overset{\eqref{a-ESO}}{\leq} f(\alpha)+\Exp[\<\nabla  f(\alpha),h_{[\hat S]}>]+\frac{1}{2\lambda n^2}\|h\|_{p\cdot v}^2\\
&=f(\alpha)+\<\nabla f(\alpha),h>_p+\frac{1}{2\lambda n^2}\|h\|_{p\cdot v}^2.
\end{align*}
Moreover, since $\psi$ is block separable, we can write
\begin{align*}
\Exp[\psi(\alpha+h_{[\hat S]})]
&\overset{\eqref{a-defofpsi}}{=}\frac{1}{n}\sum_{i=1}^n \left[\Prob(i\notin \hat S)\phi_i^*(-\alpha_i)+\Prob(i\in \hat S)\phi_i^*(-\alpha_i-h_i)\right]\\
&=\frac{1}{n}\sum_{i=1}^n\left[(1-p_i)\phi_i^*(-\alpha_i)+p_i\phi_i^*(-\alpha_i-h_i)\right].
\end{align*}
\end{proof}

Our last auxiliary result is a technical lemma for further bounding the right hand side in Lemma~\ref{l-Dalphff}.

\begin{lemma}\label{le-Palphat}
Suppose that $\hat S$ and $v\in\R^n$ satisfy Assumption~\ref{ass-ESO}. Fixing  $\alpha \in \R^{N}$ and $w\in \R^d$, let $h\in \R^{N}$ be defined by:
$$
h_i=-\theta p_i^{-1}(\alpha_i+\nabla \phi_i(A_i^\top w)),\enspace i\in[n],
$$
where $\theta$ be as in~\eqref{a-theta}. Then
\begin{equation}\begin{split}
&f(\alpha)+\<\nabla f(\alpha),h>_p+\frac{1}{2\lambda n^2}\|h\|_{p\cdot v}^2+\frac{1}{n}\sum_{i=1}^n\left[(1-p_i)\phi_i^*(-\alpha_i)+p_i\phi_i^*(-\alpha_i-h_i)\right]\\&\leq -(1-\theta)D(\alpha)-\theta\lambda g(\nabla g^*(\bar \alpha))
-\frac{1}{n}\sum_{i=1}^n\<\theta\nabla g^*(\bar \alpha), A_i\nabla \phi_i(A_i^\top w)>
+\frac{\theta}{n}\sum_{i=1}^n \phi_i^*(\nabla \phi_i(A_i^\top w))
,\end{split}\label{eq-Phalpht1}
\end{equation}
where $\bar \alpha =\frac{1}{\lambda n} A \alpha$.
\end{lemma}
\begin{proof}
Recall from~\eqref{a-defoff} that $f(\alpha)=\lambda g^*(\bar \alpha)$ and hence $\nabla f(\alpha)=\frac{1}{n} A^\top \nabla g^*(\bar \alpha)$.
Thus,
\begin{equation}\label{eq-in2}
 \begin{split}
 & f(\alpha)+\<\nabla f(\alpha),h>_p+\frac{1}{2\lambda n^2}\|h\|_{p\cdot v}^2
\\&=\lambda g^*(\bar \alpha)-\sum_{i=1}^n p_i\<\frac{1}{n}A_i^\top \nabla g^*(\bar\alpha),\theta p_i^{-1}(\alpha_i+\nabla \phi_i(A_i^\top w))>+
\frac{1}{2\lambda n^2} \|h\|^2_{p\cdot v}
\\&=(1-\theta)\lambda g^*(\bar \alpha)+\theta \lambda (g^*(\bar \alpha)-\<\nabla g^*(\bar \alpha),\bar \alpha>)
\\& \qquad-\frac{1}{n} \sum_{i=1}^n\<\theta\nabla g^*(\bar \alpha),  A_i\nabla \phi_i(A_i^\top w)>+\frac{1}{2\lambda n^2} \|h\|^2_{p\cdot v}.
\end{split}
\end{equation}
Since the functions $\phi_i$ are $(1/\gamma)$-smooth, the conjugate functions $\phi_i^*$ must be $\gamma$-strongly convex. Therefore,
\begin{align}
&\phi_{i}^*(-\alpha_{i}-h_i)
\notag
\\&=\phi_{i}^*(-(1-\theta p_i^{-1})\alpha_{i}+ \theta p_i^{-1} \nabla \phi_{i}(A_{i}^\top w))
\notag \\&\leq (1-\theta p_i^{-1})\phi_{i}^*(-\alpha_{i})+ \theta p_i^{-1}\phi_{i}^*(\nabla \phi_{i}(A_{i}^\top w))-\frac{\gamma \theta p_i^{-1}(1- \theta p_i^{-1})}{2}\| \alpha_{i} + \nabla \phi_{i}(A_{i}^\top w)\|^2
\notag \\&= (1-\theta p_i^{-1})\phi_{i}^*(-\alpha_{i})+ \theta p_i^{-1}\phi_{i}^*(\nabla \phi_{i}(A_{i}^\top w))-\frac{\gamma p_i(1- \theta p_i^{-1})}{2\theta}\| h_i\|^2,\label{a-phiital}
\end{align}
and we can write
\begin{equation}\label{eq-in1}
 \begin{split}
&\frac{1}{n}\sum_{i=1}^n[ (1-p_i)\phi_i^*(-\alpha_i)+p_i \phi_i^*(-\alpha_i-h_i)]\\
&\overset{\eqref{a-phiital}}{\leq}(1-\theta)\psi(\alpha)+\frac{\theta}{n}\sum_{i=1}^n (\phi_i^*(\nabla \phi_i(A_i^\top w)))-\frac{1}{2\lambda n^2}\sum_{i=1}^n \frac{n\lambda\gamma p_i^2(1- \theta p_i^{-1})}{\theta}\| h_i\|^2.
\end{split}
\end{equation}
Then by combining~\eqref{eq-in2} and~\eqref{eq-in1} we get:
\begin{equation*}
\begin{split}
& f(\alpha)+\<\nabla f(\alpha),h>_p+\frac{1}{2\lambda n^2}\|h\|_{p\cdot v}^2+\frac{1}{n}\sum_{i=1}^n\left[(1-p_i)\phi_i^*(-\alpha_i)+p_i\phi_i^*(-\alpha_i-h_i)\right] \\&\leq -(1-\theta)D(\alpha)-\theta\lambda g(\nabla g^*(\bar \alpha))
-\frac{1}{n}\sum_{i=1}^n\<\theta\nabla g^*(\bar \alpha), A_i\nabla \phi_i(A_i^\top w)>
+\frac{\theta}{n}\sum_{i=1}^n \phi_i^*(\nabla \phi_i(A_i^\top w))\\
&\quad +\frac{1}{2\lambda n^2}\sum_{i=1}^n \left(p_i v_i -\frac{n\lambda\gamma p_i^2(1- \theta p_i^{-1})}{\theta}\right)\| h_i\|^2.
\end{split}
\end{equation*}
It remains to notice that for $\theta$ defined in~\eqref{a-theta}, we have:
$$
p_i v_i \leq \frac{n\lambda\gamma p_i^2(1- \theta p_i^{-1})}{\theta},\enspace \forall i \in[n].
$$
\end{proof}

\subsection{Proof of  Theorem~\ref{th-mainth}}

Let $t\geq 1$. Define $h^t=(h_1^t,\dots,h_n^t)\in \R^{N}$ by:
$$
h_i^t=-\theta p_i^{-1}(\alpha^{t-1}_i+\nabla \phi_i(A_i^\top w^t)),\enspace i\in[n]
$$
and $\kappa^t=(\kappa_1^t,\cdots,\kappa_n^t)$ by:
$$
\kappa_i^t=\arg\max_{\Delta\in \R^m}
\left[-\phi_i^*(-(\alpha_i^{t-1}+\Delta ))-\nabla g^*(\bar\alpha^{t-1})^\top A_i \Delta-\frac{v_i \|\Delta\|  ^2}{2\lambda n}\right],\enspace \forall i\in [n].
$$
If we use Option \rom{1} in Algorithm~\ref{algo:primaldual}, then $
\alpha^t=\alpha^{t-1}+\kappa^t_{[\hat S]}$. If we use Option \rom{2} in Algorithm~\ref{algo:primaldual},
then we have $\alpha^t=\alpha^{t-1}+h^t_{[\hat S]}$. In both cases,
by Lemma~\ref{l-Dalphff}:
\begin{equation*}
\begin{split}
&\Exp_t[-D(\alpha^t)]\\&\leq f(\alpha^{t-1})+\<\nabla f(\alpha^{t-1}),h^t>_p+\frac{1}{2\lambda n^2}\|h^t\|_{p\cdot v}^2+\frac{1}{n}\sum_{i=1}^n\left[(1-p_i)\phi_i^*(-\alpha^{t-1}_i)+p_i\phi_i^*(-\alpha^{t-1}_i-h^t_i)\right].\end{split}
\end{equation*}
We now apply Lemma~\ref{le-Palphat} to further bound the last term and obtain:
\begin{equation}\label{eq-Dualbound}
\begin{split}
\Exp_t[-D(\alpha^t)] & \leq -(1-\theta)D(\alpha^{t-1})-\theta\lambda g(\nabla g^*(\bar \alpha^{t-1})) \\& \quad
-\frac{1}{n}\sum_{i=1}^n\<\theta\nabla g^*(\bar \alpha^{t-1}), A_i\nabla \phi_i(A_i^\top w^t)>
+\frac{\theta}{n}\sum_{i=1}^n \phi_i^*(\nabla \phi_i(A_i^\top w^t)).\end{split}
\end{equation}
By  convexity of $g$,
\begin{equation}\label{eq-pbound}\begin{split}
P(w^t)&=\frac{1}{n}\sum_{i=1}^n \phi_i(A_i^\top w^t) + \lambda g((1-\theta)w^{t-1}+\theta \nabla g^*(\bar \alpha^{t-1}))
\\& \leq \frac{1}{n}\sum_{i=1}^n \phi_i(A_i^\top w^t) +(1-\theta)\lambda g(w^{t-1})+\theta \lambda g(\nabla g^*(\bar \alpha^{t-1})).
\end{split}
\end{equation}
By combining~\eqref{eq-Dualbound} and~\eqref{eq-pbound} we get:
\begin{equation}\label{eq-erdf}\begin{split}
\Exp_t[P(w^t)-D(\alpha^t)] & \leq \frac{1}{n}\sum_{i=1}^n \phi_i(A_i^\top w^t) +(1-\theta)\lambda g(w^{t-1}) -(1-\theta)D(\alpha^{t-1})
\\& \qquad-\frac{1}{n} \sum_{i=1}^n\<\theta\nabla g^*(\bar \alpha^{t-1}), A_i\nabla \phi_i(A_i^\top w^t)>
+\frac{\theta}{n}\sum_{i=1}^n \phi_i^*(\nabla \phi_i(A_i^\top w^t))
\\&=(1-\theta)(P(w^{t-1})-D(\alpha^{t-1}))+\frac{1}{n}\sum_{i=1}^n (\phi_i(A_i^\top w^t)-(1-\theta)\phi_i(A_i^\top w^{t-1}))\\ &\quad 
-\frac{1}{n}\sum_{i=1}^n \<\theta \nabla g^*(\bar \alpha^{t-1}),  A_i \nabla \phi_i(A_i^\top w^t) >+\frac{\theta}{n} \sum_{i=1}^n \phi_{i}^*(\nabla \phi_{i}(A_{i}^\top w^t)).
\end{split}
\end{equation}
Note that $\theta \nabla g^*(\bar \alpha^{t-1})=w^t-(1-\theta)w^{t-1}$ and $ 
 \phi_{i}^*(\nabla \phi_{i}(A_{i}^\top w^t))=\<\nabla \phi_i(A_i^\top w^t), A_i^\top w^t >-\phi_i(A_i^\top w^t)$. Finally, we  plug these two inequalities into~\eqref{eq-erdf} and obtain:
\begin{equation*}
\begin{split}
\Exp_t[P(w^t)-D(\alpha^t)]
\leq& (1-\theta)(P(w^{t-1})-D(\alpha^{t-1})) +\frac{1}{n}\sum_{i=1}^n (\phi_i(A_i^\top w^t)-(1-\theta)\phi_i(A_i^\top w^{t-1}))\\ &\quad
-\frac{1}{n}\sum_{i=1}^n \<A_i^\top w^t-(1-\theta)A_i^\top w^{t-1},   \nabla \phi_i(A_i^\top w^t) >\\& \quad \quad+\frac{\theta}{n} \sum_{i=1}^n \big(\<\nabla \phi_i(A_i^\top w^t), A_i^\top w^t >-\phi_i(A_i^\top w^t)\big)
\\=& (1-\theta)(P(w^{t-1})-D(\alpha^{t-1})) +\frac{1-\theta}{n}\sum_{i=1}^n (\phi_i(A_i^\top w^t)-\phi_i(A_i^\top w^{t-1}))\\ &\quad
-\frac{1-\theta}{n}\sum_{i=1}^n \<A_i^\top w^t-A_i^\top w^{t-1},   \nabla \phi_i(A_i^\top w^t) >\\
=&(1-\theta)(P(w^{t-1})-D(\alpha^{t-1}))\\ &\quad+\frac{1-\theta}{n}\sum_{i=1}^n \big[\phi_i(A_i^\top w^t)-\phi_i(A_i^\top w^{t-1})+ \<A_i^\top w^{t-1}-A_i^\top w^{t},   \nabla \phi_i(A_i^\top w^t) > \big]
\\\leq & (1-\theta)(P(w^{t-1})-D(\alpha^{t-1})),
\end{split}
\end{equation*}
where the last inequality follows from the convexity of  $\phi_i$.

\section{Experimental Results} \label{sec:experiments}

In~\cite{SDCA} and~\cite{ASDCA}, the reader can find an extensive list of popular machine learning problems to which Prox-SDCA can be applied. Sharing the same primal-dual formulation, our algorithm can also be specified and applied to those applications, including Ridge regression, SVM, Lasso, 
logistic regression and multiclass prediction. 

We focus our numerical experiments on the L2-regularized linear SVM problem with smoothed hinge loss or squared hinge loss. These problems are described  in detail in Section~\ref{sec:apps}. The three main messages that we draw from the numerical experiments are:
\begin{itemize}
\item Importance sampling does improve the convergence for certain datasets;
 \item Quartz specialized to serial samplings is comparable to Prox-SDCA in practice;
\item The theoretical speedup factor is an almost exact predictor of the actual speedup  (in terms
of iteration complexity).
\end{itemize}
We performed the experiments on several real world large datasets, of various dimensions $n$, $d$ and sparsity. The details of the dataset characteristics are provided in Table~\ref{tab-num}. In all our experiments we used Option \rom{1}, which we found to be better in practice.

\begin{table}[ht]
\centering
 \begin{tabular}{|c|c|c|c|}
\hline
  Dataset & \# Training size $n$ & \# features $d$ & Sparsity (\# nnz$/(nd)$)\\
\hline 
astro-ph & 29,882 & 99,757 & 0.08\% \\
\hline 
CCAT & 781,265 & 47,236 & 0.16\% \\
\hline
cov1 & 522,911 & 54 & 22.22\%\\
\hline 
w8a & 49,749 & 300 &  3.91\% \\
\hline ijcnn1 & 49,990 & 22 & 59.09\% \\
\hline 
webspam & 350,000 & 254  & 33.52\% \\\hline
 \end{tabular}
\caption{Datasets used in our experiments.}
\label{tab-num}
\end{table}

\subsection{Applications} \label{sec:apps}

\paragraph{Smooth hinge loss with $L_2$ regularizer.} We specify Quartz to the linear  Support Vector Machine (SVM) problem with smoothed hinge loss and $L_2$ regularizer:
\begin{equation*}
\min_{w\in \R^d}\;\; P(w)\eqdef\frac{1}{n}\sum_{i=1}^n \phi_i(y_iA_i^\top w) + \lambda g(w),
\end{equation*}
where 
\begin{equation}\label{d-1}
 \phi_i(a)=\left\{\begin{array}
{ll} 0 & a\geq 1\\ 1-a-\gamma/2 & a\leq 1-\gamma \\ \displaystyle\frac{(1-a)^2}{2\gamma} & \mathrm{otherwise.}
\end{array}\right. ,\enspace \forall a\in \R
\end{equation}
and 
\begin{align}\label{g-app-1}
g(w)=\frac{1}{2}\|w\|^2,\qquad w\in \R^d.
\end{align}
Here  $y_i\in\{\pm 1\}$ is the label of the example $A_i\in\R^d$. One can get rid of the labels by redefining
$A_i$ to $y_iA_i$.
Note that $\phi_i$ defined by~\eqref{d-1} is  $1/\gamma$-smooth and $g$ defined by~\eqref{g-app-1} is 1-strongly convex.
 In this special case,
Option \rom{1} in Algorithm~\ref{algo:primaldual} has a closed form solution: 
$$
\Delta \alpha_i^t=\max\left\{  -\alpha_i^{t-1},\min\left\{ 1-\alpha_i^{t-1},\frac{1-y_iA_i^\top \bar \alpha^{t-1}-\gamma \alpha_i^{t-1}}{v_i/(\lambda n)+\gamma} \right\}\right\}.
$$
See also~\cite[Section 5.6]{ASDCA}.

\paragraph{Squared hinge loss with $L_2$ regularizer.} We now specify Quartz to the linear Support Vector Machine (SVM) problem with squared hinge loss based and $L_2$ regularizer:
\begin{equation*}
\min_{w\in \R^d}\;\; P(w)\eqdef\frac{1}{n}\sum_{i=1}^n \phi_i(y_iA_i^\top w) + \lambda g(w),
\end{equation*}
where 
\begin{equation}\label{d-2}
 \phi_i(a)=\frac{([1-a]_+)^2}{2\gamma} ,\qquad  a\in \R
\end{equation}
and 
\begin{align}\label{g-app-2}
g(w)=\frac{1}{2}\|w\|^2,\qquad  w\in \R^d.
\end{align}
Note that $\phi_i$ defined by~\eqref{d-2} is  $1/\gamma$-smooth 
and $g$ defined by~\eqref{g-app-2} is 1-strongly convex. In this special case,
Option \rom{1} in Algorithm~\ref{algo:primaldual} has a closed form solution:
$$
\Delta \alpha_i^{t}=\max\left\{\frac{1-y_i A_i^\top \bar\alpha^{t-1}-\gamma \alpha_i^{t-1}}{\gamma+v_i/(\lambda n)},-\alpha_i^{t-1}\right\}.
$$ See also~\cite[Section 5.3]{IProx-SDCA}.

\begin{figure}[ht]
\centering
\subfigure[cov1; $n=522,911$; $\lambda=1$e-06]{
     \includegraphics[width=0.3\textwidth]{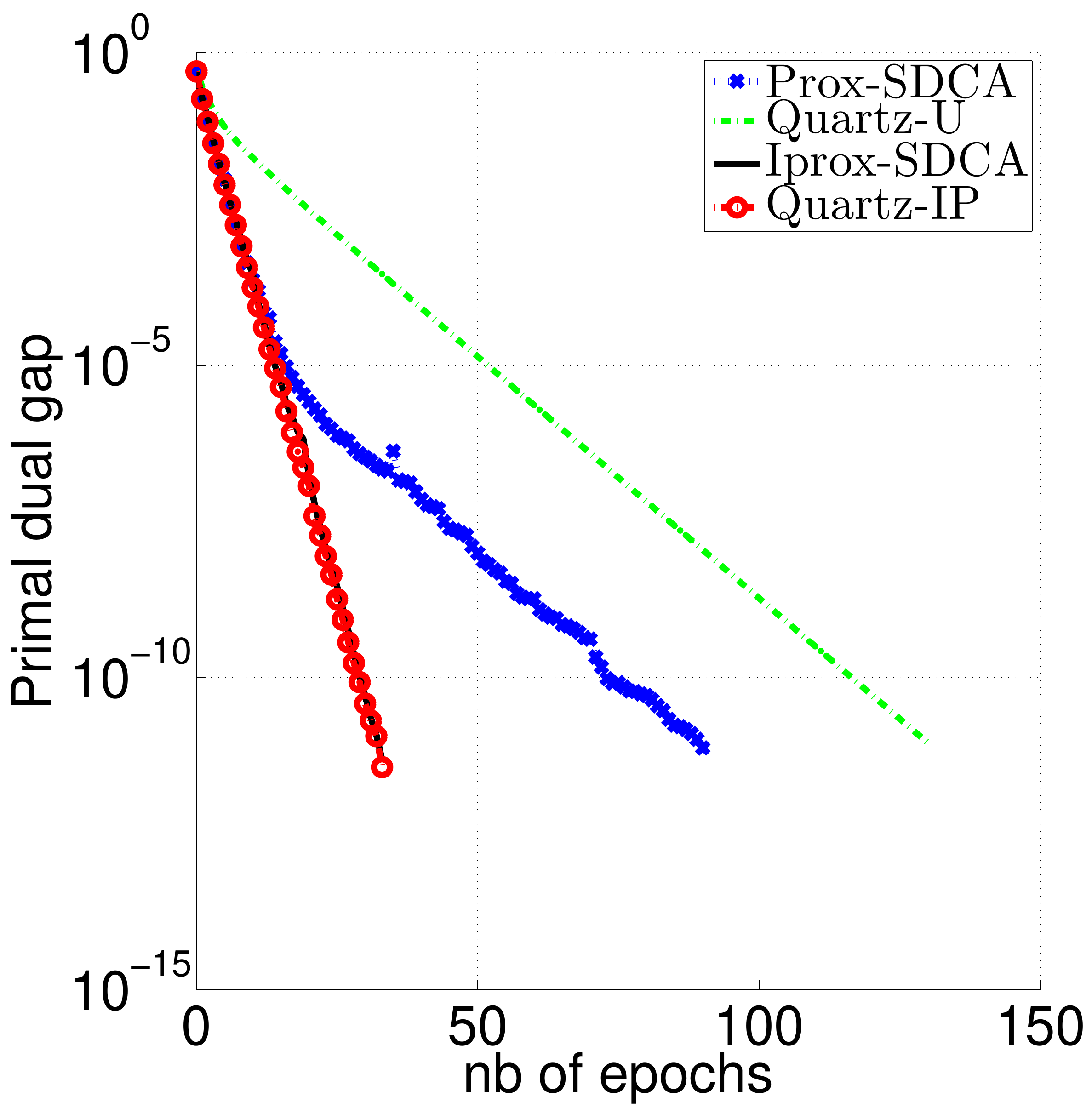}
    \label{figun:subfig1}
}
\subfigure[w8a; $n=49,749$; $\lambda=1$e-05]{
    \includegraphics[width=0.3\textwidth]{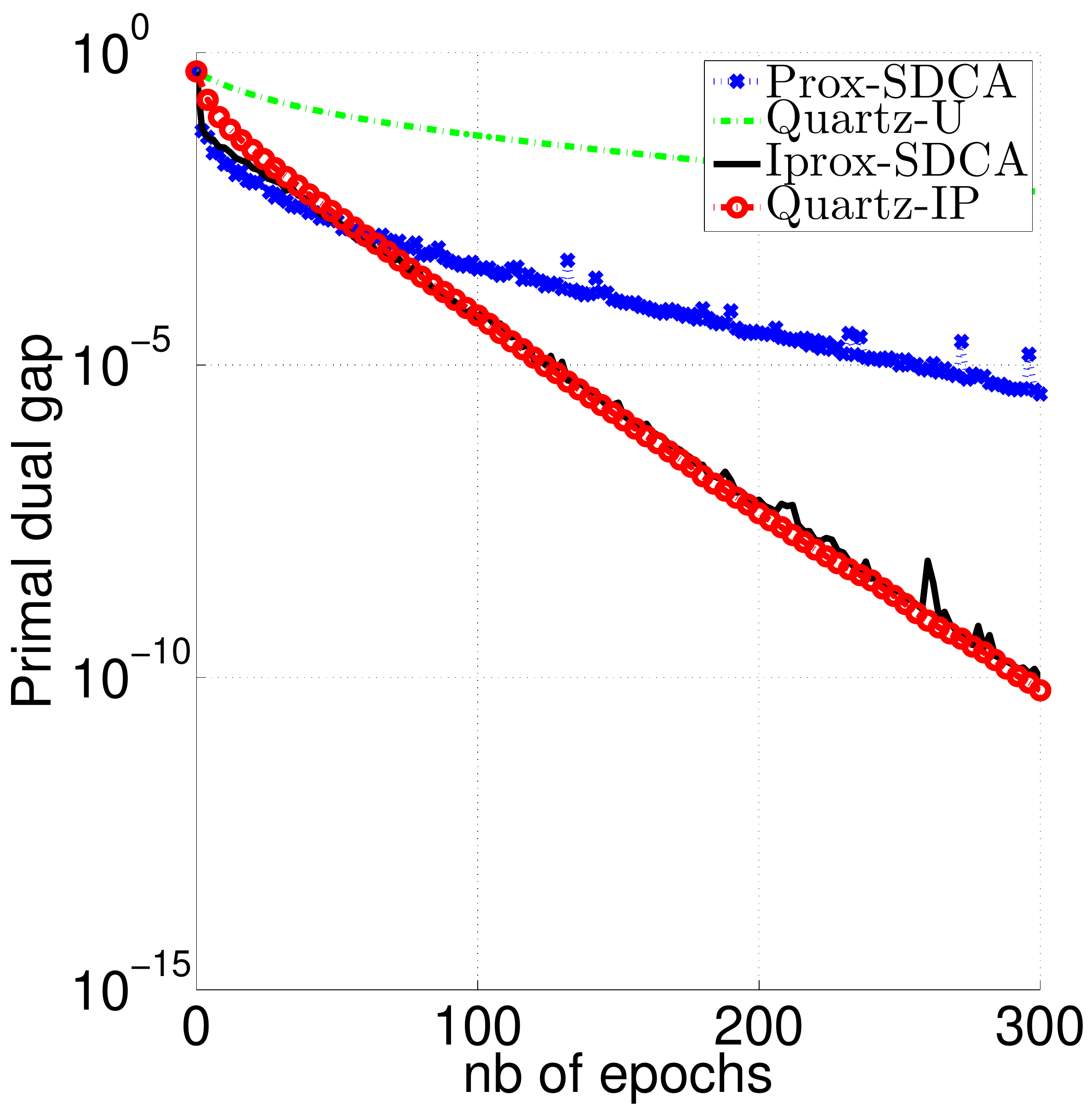}
    \label{figun:subfig2}
}
\subfigure[ijcnn1; $n=49,990$;  $\lambda=1$e-05]{
    \includegraphics[width=0.3\textwidth]{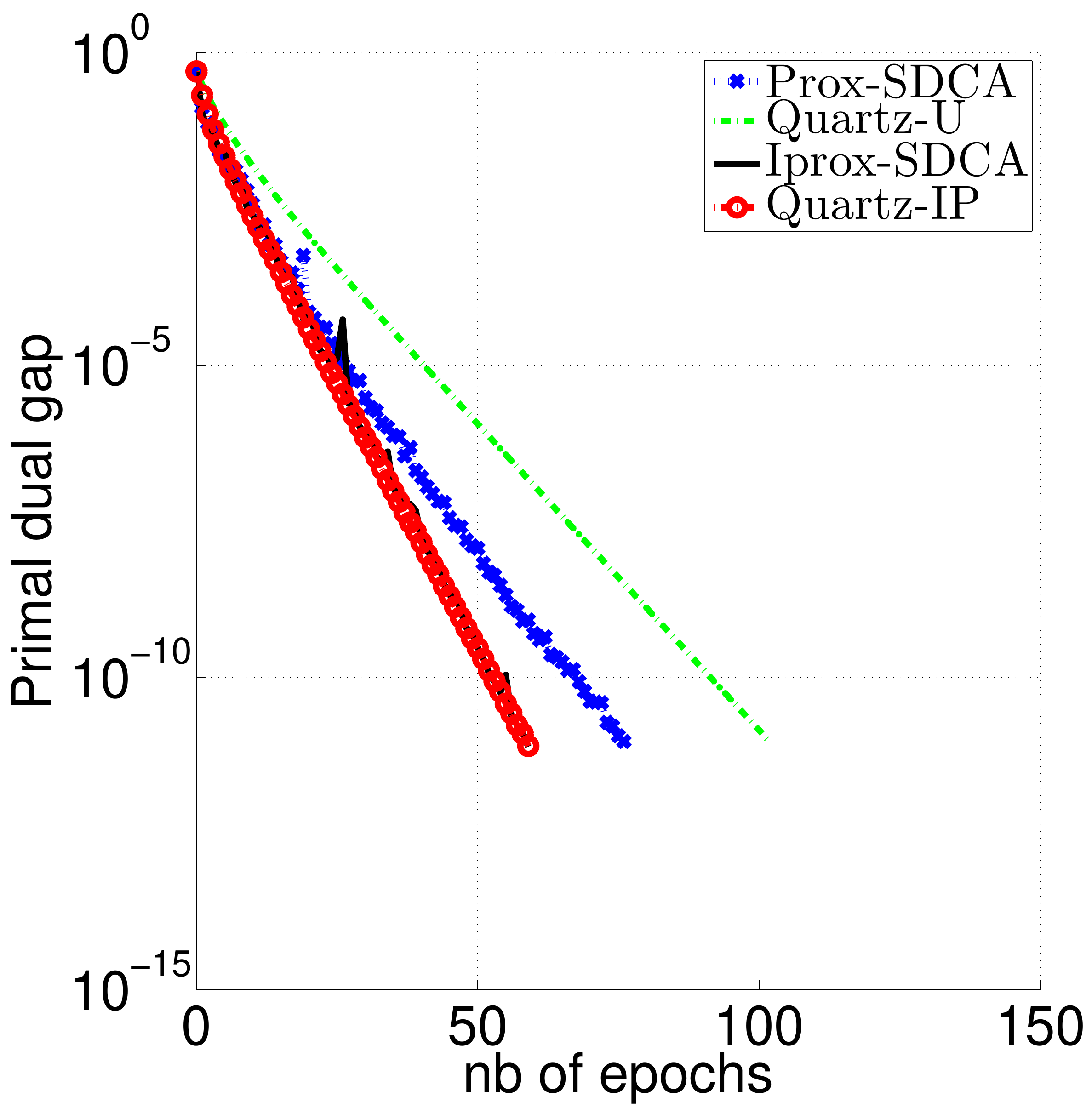}
    \label{figun:subfig3}
}
\subfigure[cov1; $n=522,911$; $\lambda=1$e-06]{
     \includegraphics[width=0.3\textwidth]{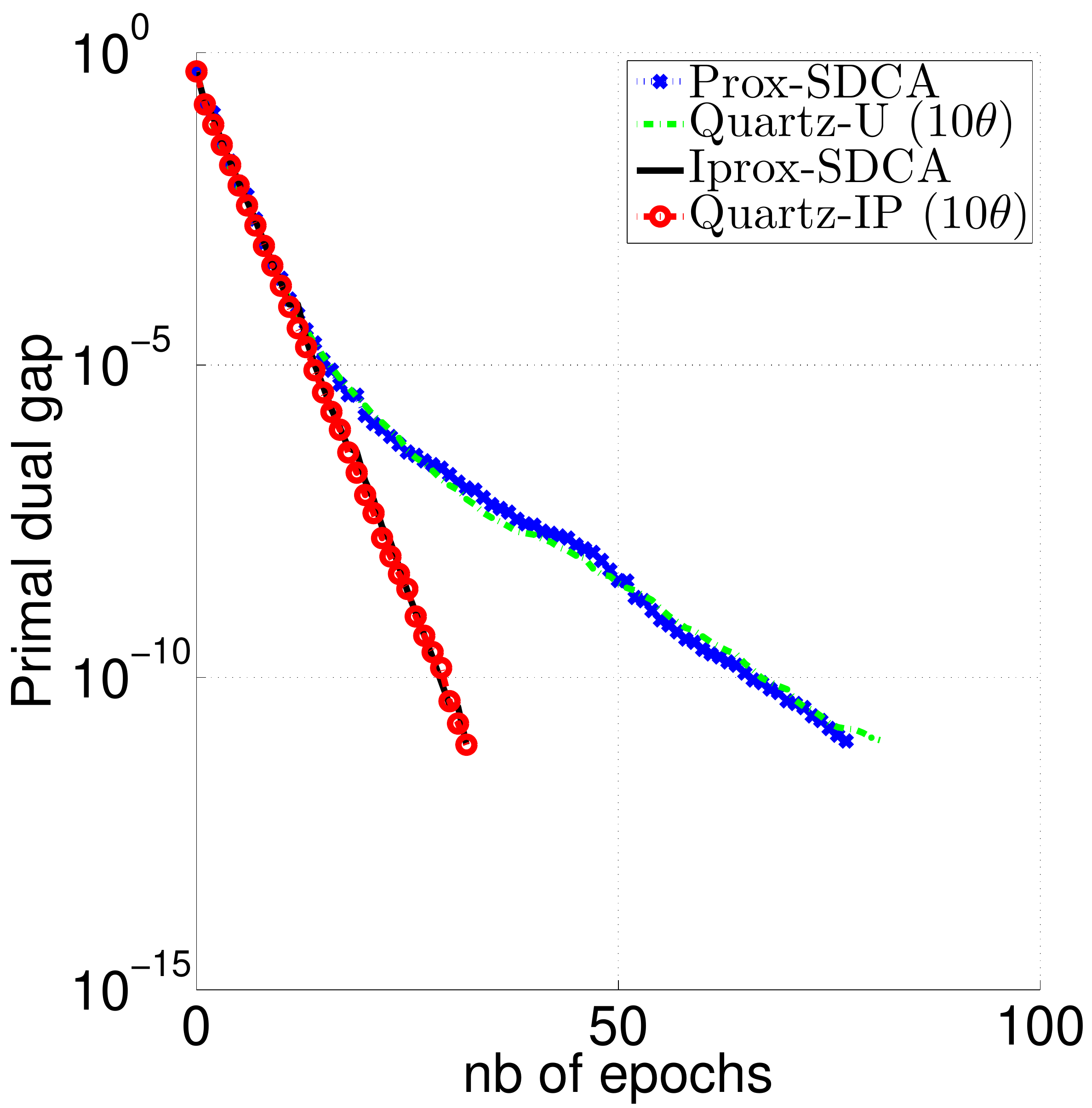}
    \label{figun:subfig4}
}
\subfigure[w8a; $n=49,749$; $\lambda=1$e-05]{
    \includegraphics[width=0.3\textwidth]{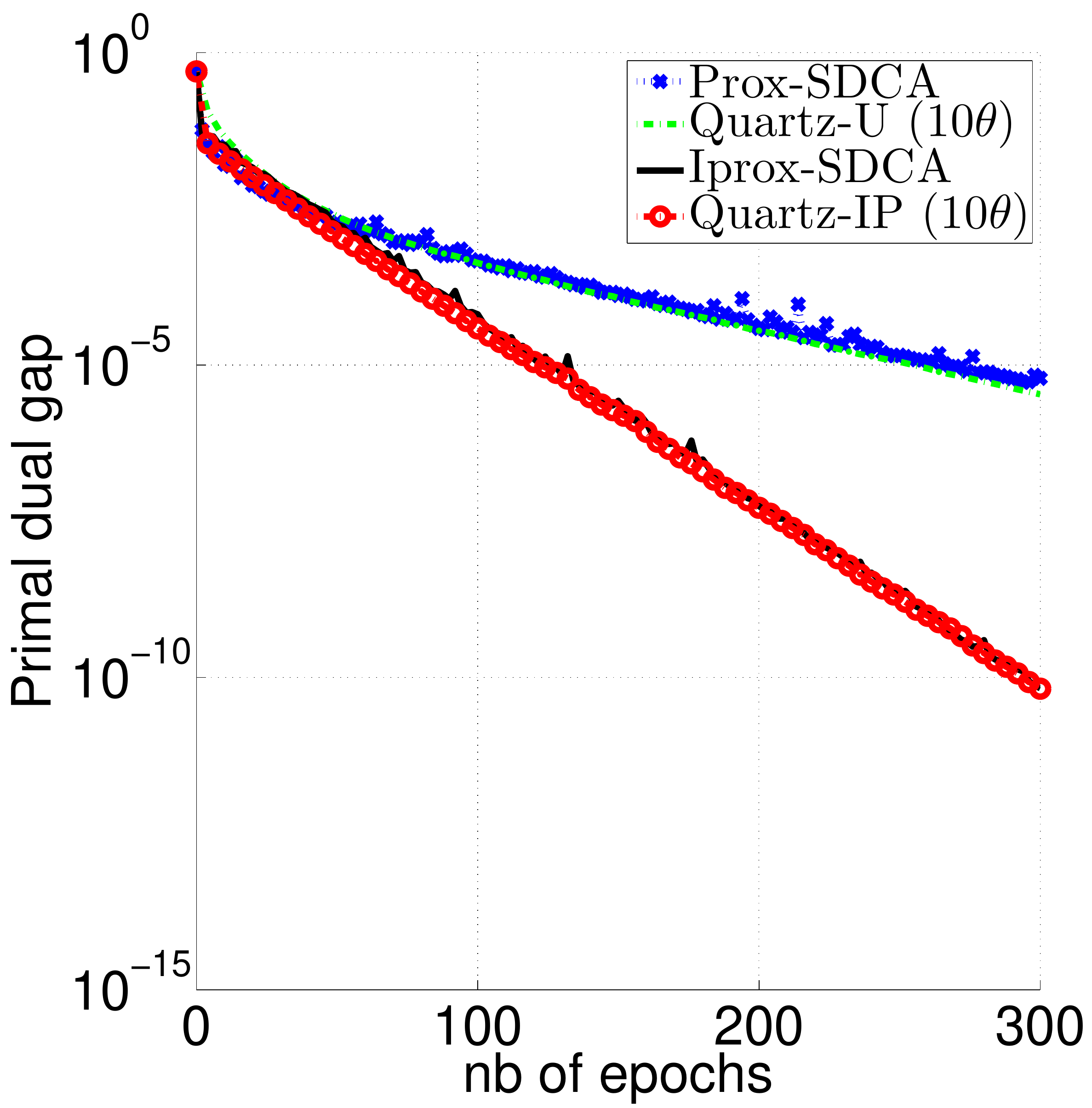}
    \label{figun:subfig5}
}
\subfigure[ijcnn1; $n=49,990$; $\lambda=1$e-05]{
    \includegraphics[width=0.3\textwidth]{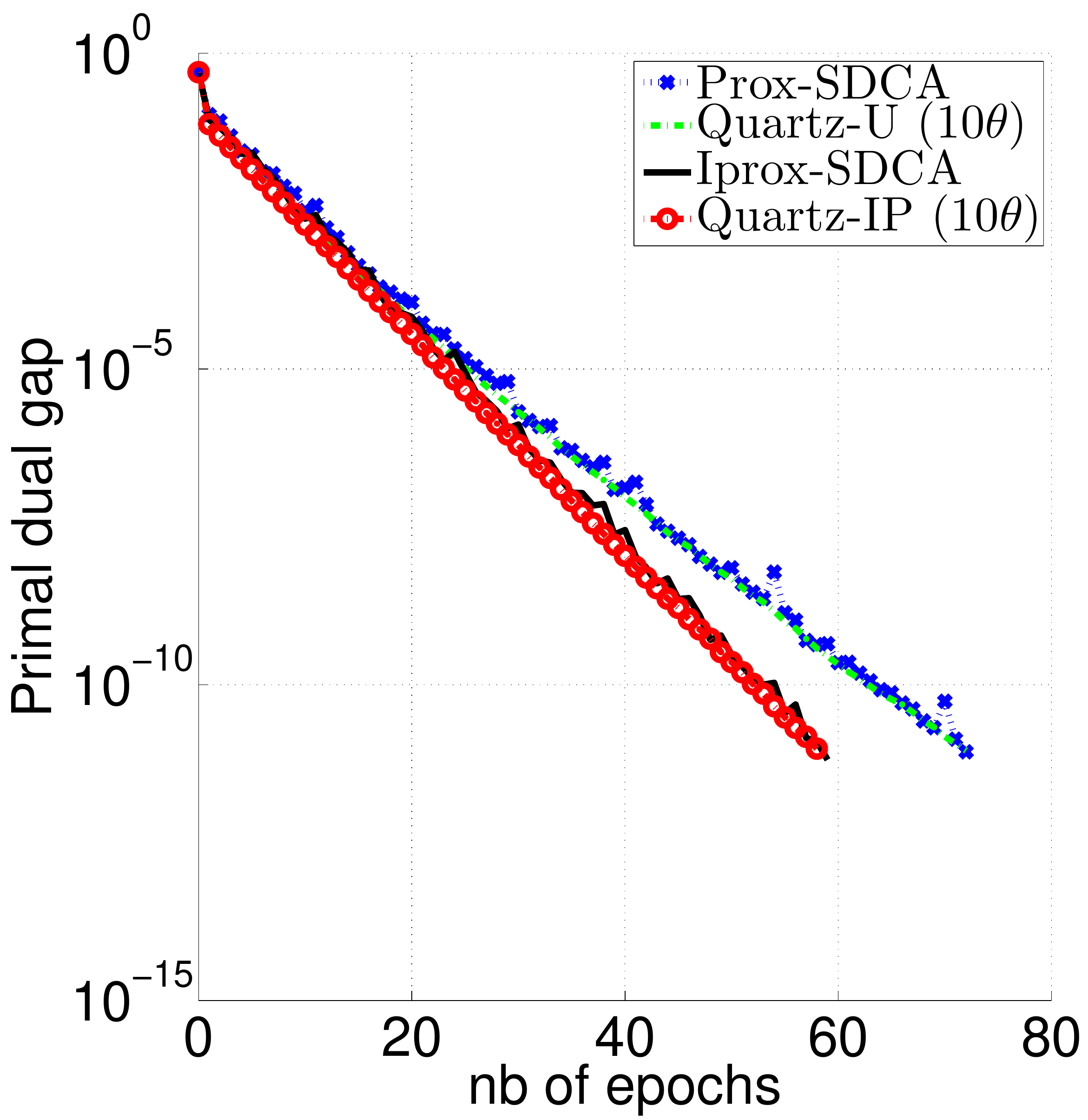}
    \label{figun:subfig6}
}
\subfigure[cov1; $n=522,911$; $\lambda=1$e-06]{
     \includegraphics[width=0.3\textwidth]{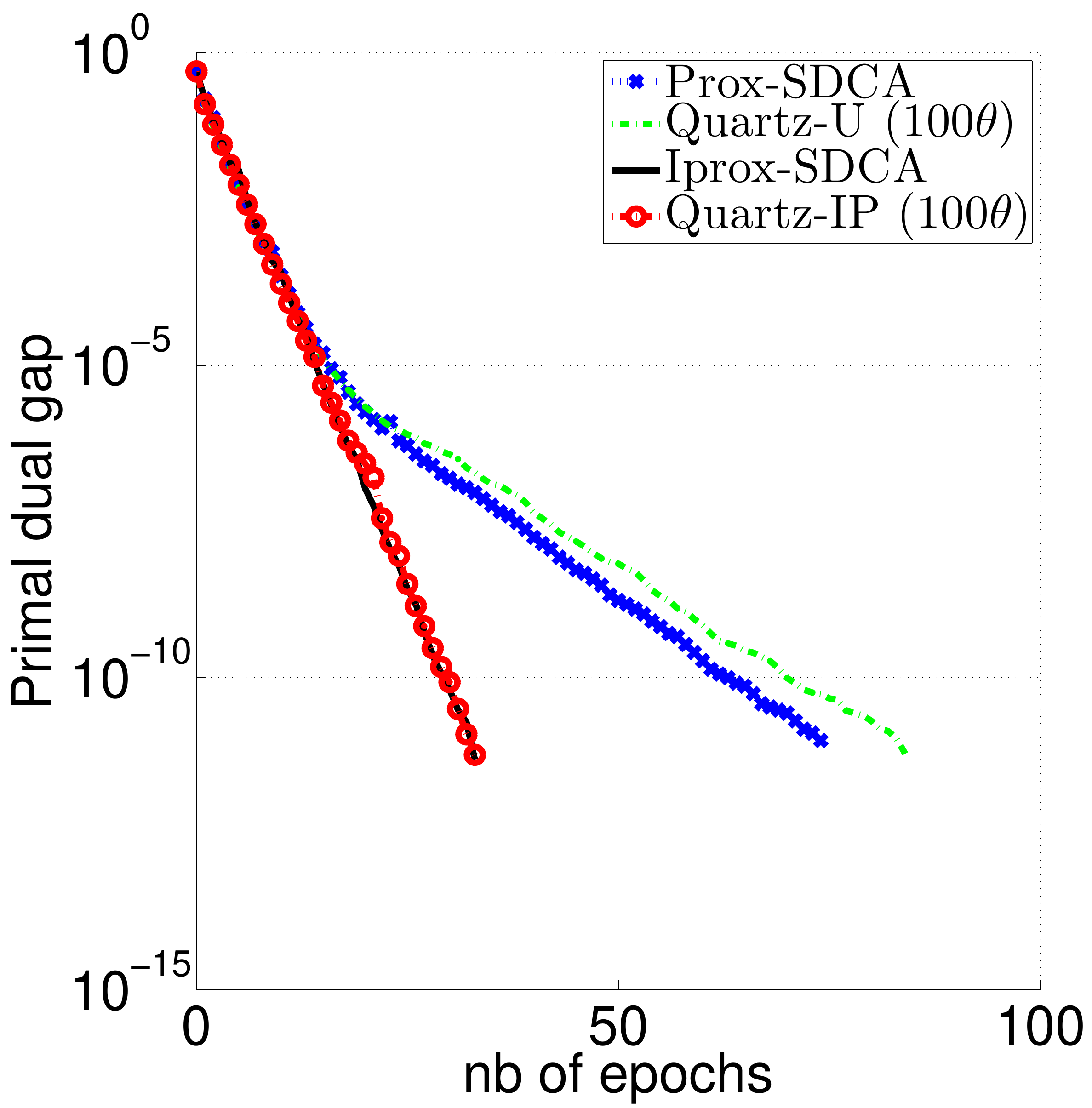}
    \label{figun:subfig7}
}
\subfigure[w8a; $n=49,749$; $\lambda=1$e-05]{
    \includegraphics[width=0.3\textwidth]{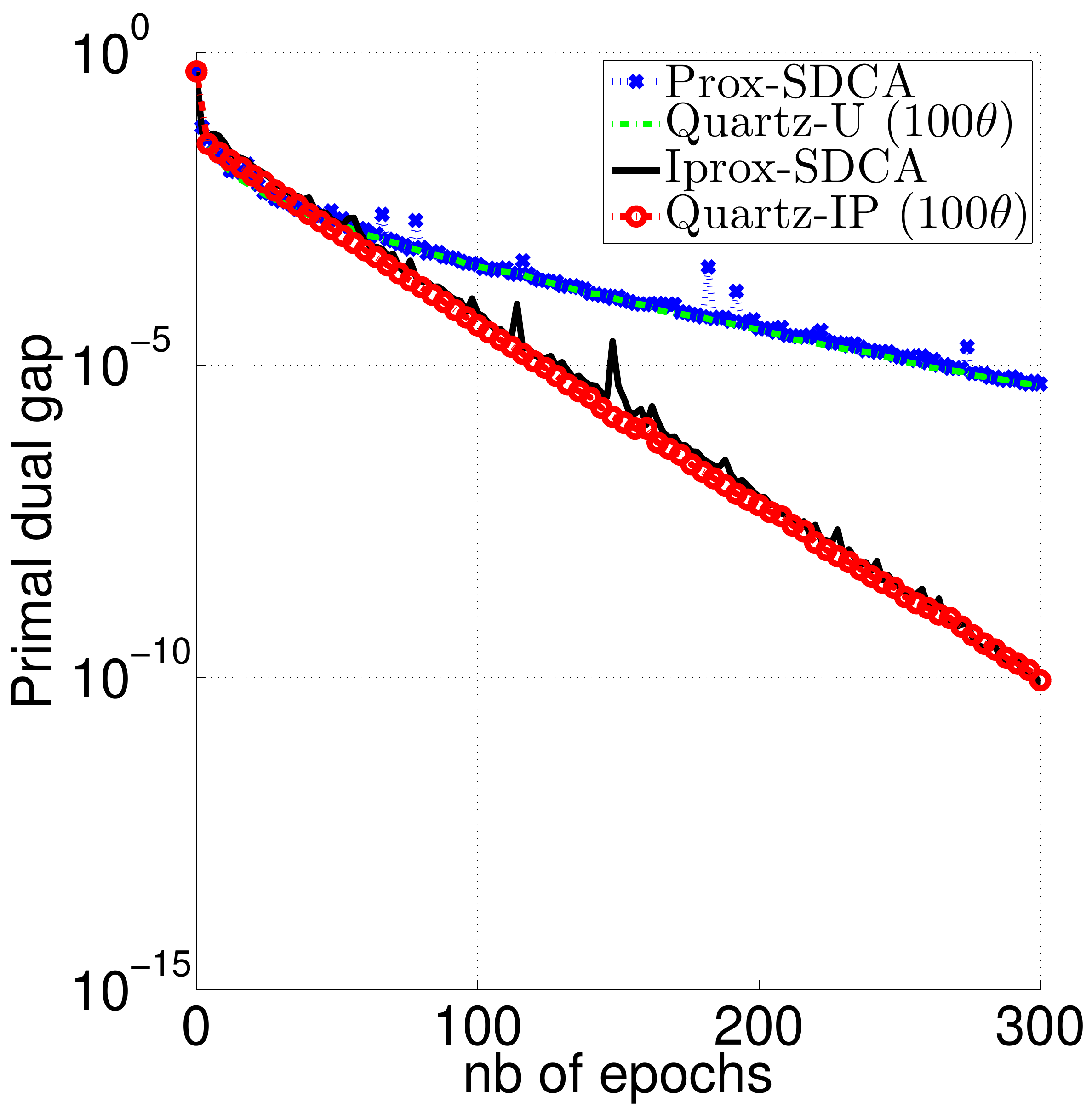}
    \label{figun:subfig8}
}
\subfigure[ijcnn1; $n=49,990$; $\lambda=1$e-05]{
    \includegraphics[width=0.3\textwidth]{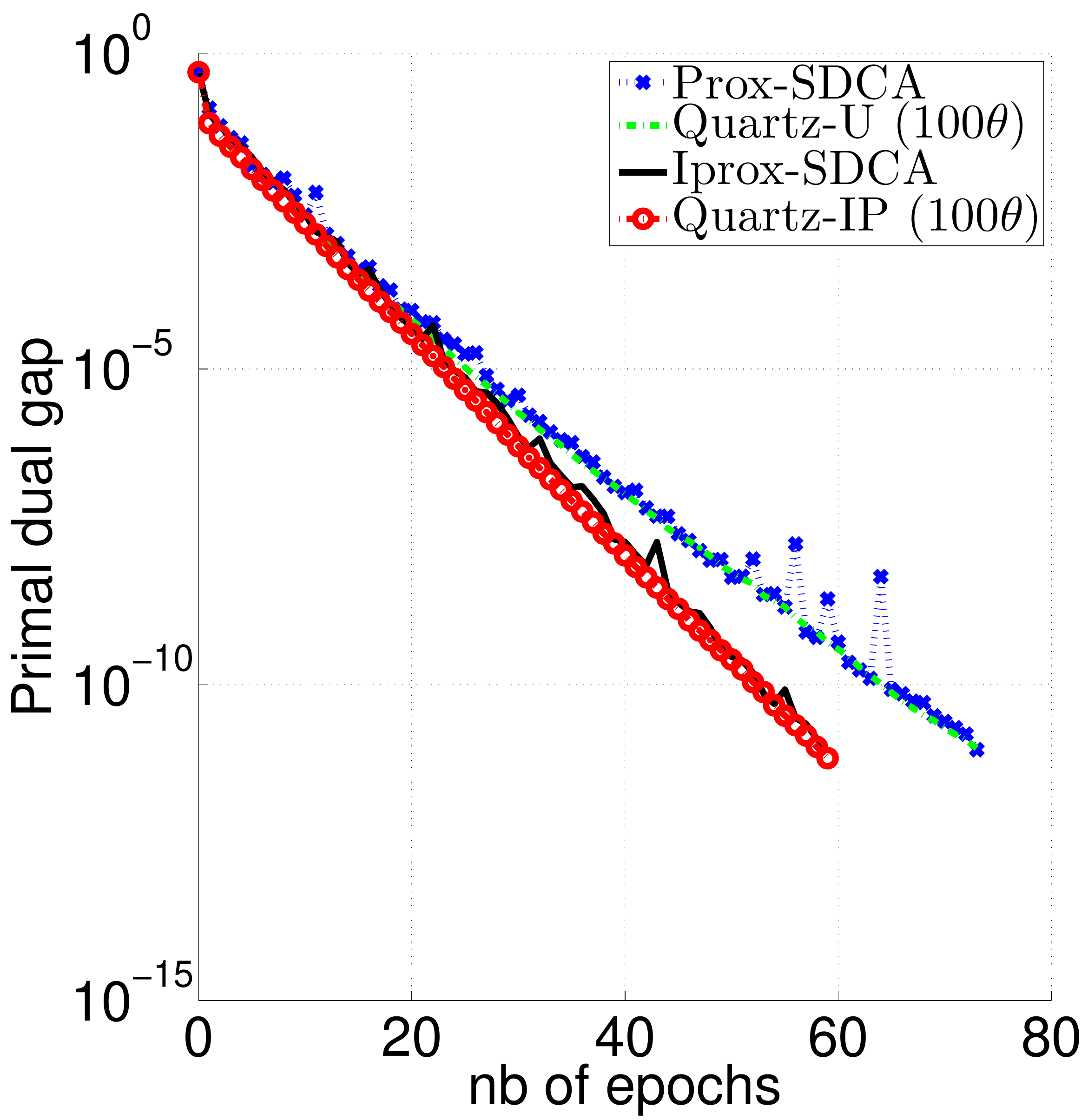}
    \label{figun:subfig9}
}
\caption[Optional caption for list of figures]{Comparison of Quartz-U (uniform sampling),
Quartz-IP (optimal importance sampling), Prox-SDCA (uniform sampling) and Iprox-SDCA (optimal importance sampling). In Figure~\ref{figun:subfig4},~\ref{figun:subfig5} and~\ref{figun:subfig6}, we used aggressive primal update: $w^{t}=(1-10\theta)w^{t-1}+10\theta \nabla g^*(\bar \alpha^{t-1})$. In Figure~\ref{figun:subfig7},~\ref{figun:subfig8} and~\ref{figun:subfig9}, we used aggressive primal update: $w^{t}=(1-100\theta)w^{t-1}+100\theta \nabla g^*(\bar \alpha^{t-1})$. The loss function is the squared hinge loss. The regularizer is the $L_2$-regularizer. }
\label{fig:ex1IU}
\end{figure}

\subsection{Quartz and SDCA for uniform and importance sampling}

In this section we compare four algorithms: 
\begin{itemize}
 \item Quartz-U: Quartz specialized to uniform serial sampling;
\item Prox-SDCA~\cite{SDCA,ASDCA}: proximal stochastic dual coordinate ascent with uniform sampling;
\item Quartz-IP: Quartz specialzed to importance sampling;
\item Iprox-SDCA~\cite{IProx-SDCA}: proximal stochastic dual coordinate ascent with importance sampling;
\end{itemize}
on three datasets: cov1, w8a and ijcnn1. We consider the $L2$-regularized linear SVM problem using  squared hinge loss,
as described in Section~\ref{sec:apps}. The value of $\gamma$ is set to be 1 and the value of 
$\lambda$ varies between the datasets: $10^{-5}$ for w8a and ijcnn1 and $10^{-6}$ for cov1, whose
number of training examples $n$ is 10 times larger than the other datasets. The results are shown in Figure~\ref{fig:ex1IU}.

\paragraph{Utility of importance sampling.} If we compare Quartz-U with Quartz-IP, it is clear that importance 
sampling provides better convergence rate than uniform sampling on the datasets that we tested.

\paragraph{Similarity between Quartz-IP and Iprox-SDCA.} In all the experiments, Quartz-IP shows an almost identical 
convergence behaviour to that of Iprox-SDCA. 

\paragraph{Conservative primal update in Quartz.} While Quartz-IP has the same practical convergence rate as
Iprox-SDCA, Quartz-U appears to be somewhat slower than Prox-SDCA in practice. One possible explanation is that the primal update in Quartz,
\begin{align}
w^t=(1-\theta)w^{t-1}+\theta \nabla g^*(\bar \alpha^{t-1}),
\end{align}
is too conservative. Indeed, since the optimal solution satisfies 
$w^*=\nabla g^*(\bar \alpha^*)$, larger $\theta$ leads to faster convergence on the primal problem 
when the dual variable $\alpha^{t-1}$ is close
to the optimal solution $\alpha^*$. To confirm this, we tested two more aggressive primal update rules:  Quartz-10$\theta$ and Quartz-100$\theta$ which change 
 the primal update to:
\begin{align*}
w^t=(1-10\theta)w^{t-1}+10\theta \nabla g^*(\bar \alpha^{t-1}),
\end{align*}
and 
\begin{align*}
w^t=(1-100\theta)w^{t-1}+100\theta \nabla g^*(\bar \alpha^{t-1}),
\end{align*}
respectively. The results are displayed in Figure~\ref{figun:subfig4},~\ref{figun:subfig5},~\ref{figun:subfig6},~\ref{figun:subfig7},~\ref{figun:subfig8} and~\ref{figun:subfig9}. It is clear that with just a slightly more aggressive primal update rule than the one sanctioned by our theory, Quartz-U achieves similar
practical convergence as Prox-SDCA. Recall that
the primal update in Prox-SDCA is
$w^t=\nabla g^*(\bar \alpha^{t-1})$. Notice also that the parameter $\theta$ defined by~\eqref{a-theta} is less than $1/n$, hence close to 0. 
Therefore, there is still a difference in the primal update rules between Quartz-100$\theta$ and Prox-SDCA.

\subsection{Mini-batch experiments}

In this section we demonstrate that the {\em theoretical speedup factor}  of Quartz specialized to sampling $\hat{S}$
is a very good predictor of the {\em practical speedup factor}, defined as: 
$$
\frac{\#\mathrm{~of~iterations~till~}\epsilon\mathrm{~primal~dual~gap~is~found~by~Quartz~specialized~to~serial~uniform~sampling} }
{\#\mathrm{~of~iterations~till~}\epsilon\mathrm{~primal~dual~gap~is~found~by~Quartz~specialized~to~}\hat S}.
$$

We focus on the problem of training $L2$-regularized linear SVMs with smoothed hinge loss $(\gamma=1)$, described in Section~\ref{sec:apps}. In the   experiments we chose  $\epsilon = 10^{-11}$. 

\begin{figure}[ht]
\centering
\subfigure[astro\_ph; sparsity: 0.08\%;]{
     \includegraphics[width=0.3\textwidth]{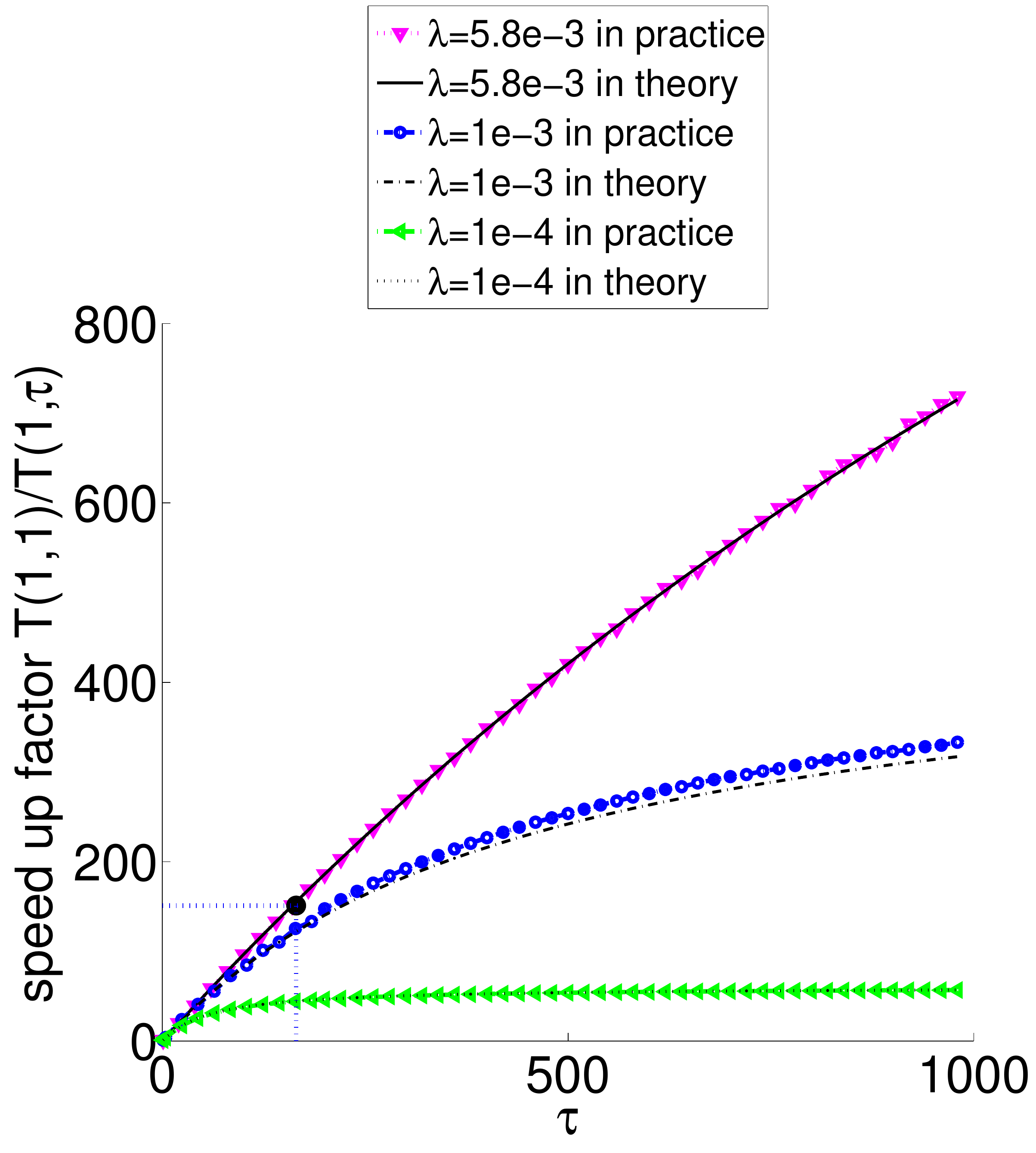}
    \label{figtau:subfig3}
}
\subfigure[CCAT; sparsity: 0.16\%;]{
    \includegraphics[width=0.3\textwidth]{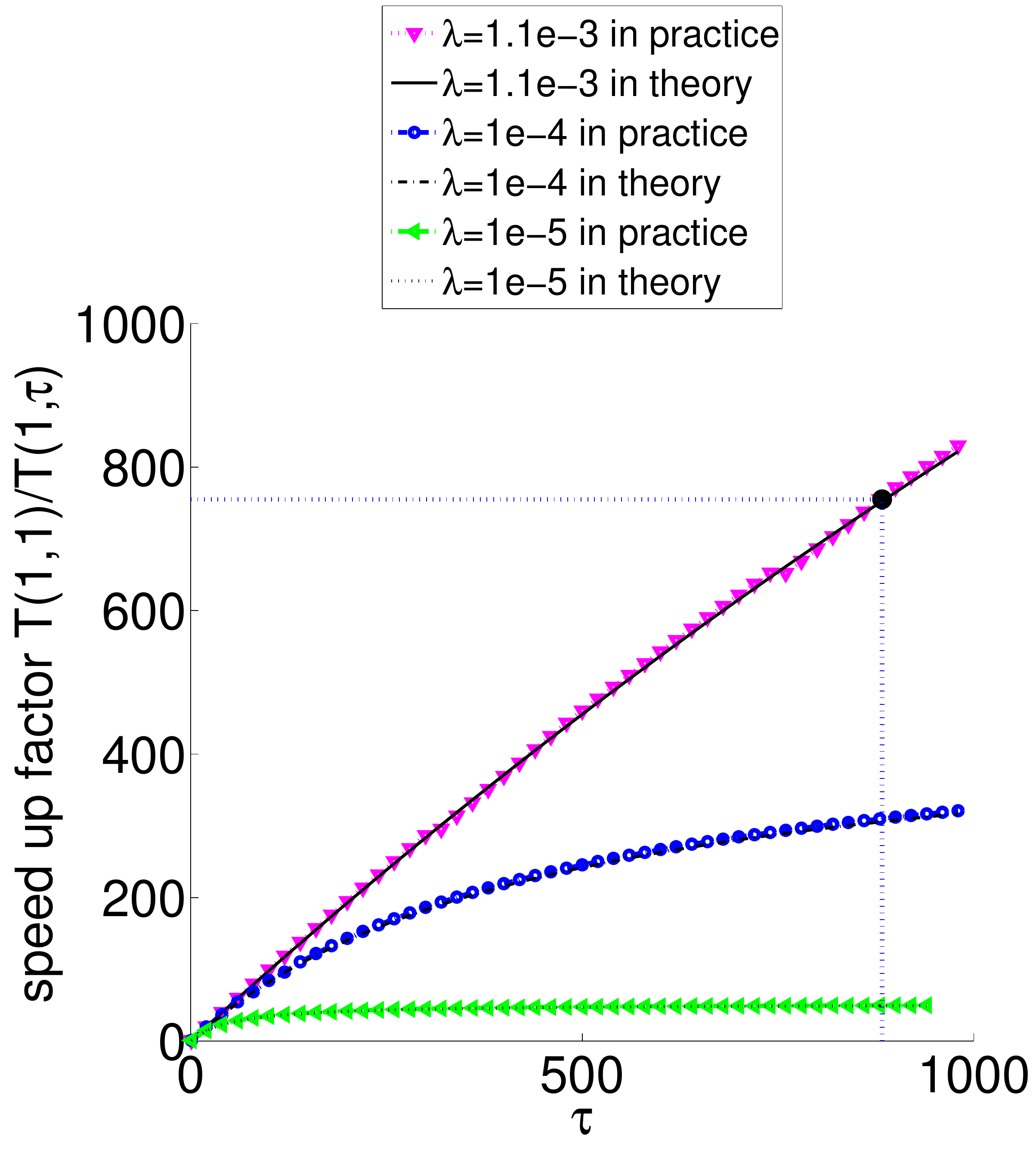}
    \label{figtau:subfig1}
}
\subfigure[cov1;  sparsity: 22.22\%;]{
    \includegraphics[width=0.3\textwidth]{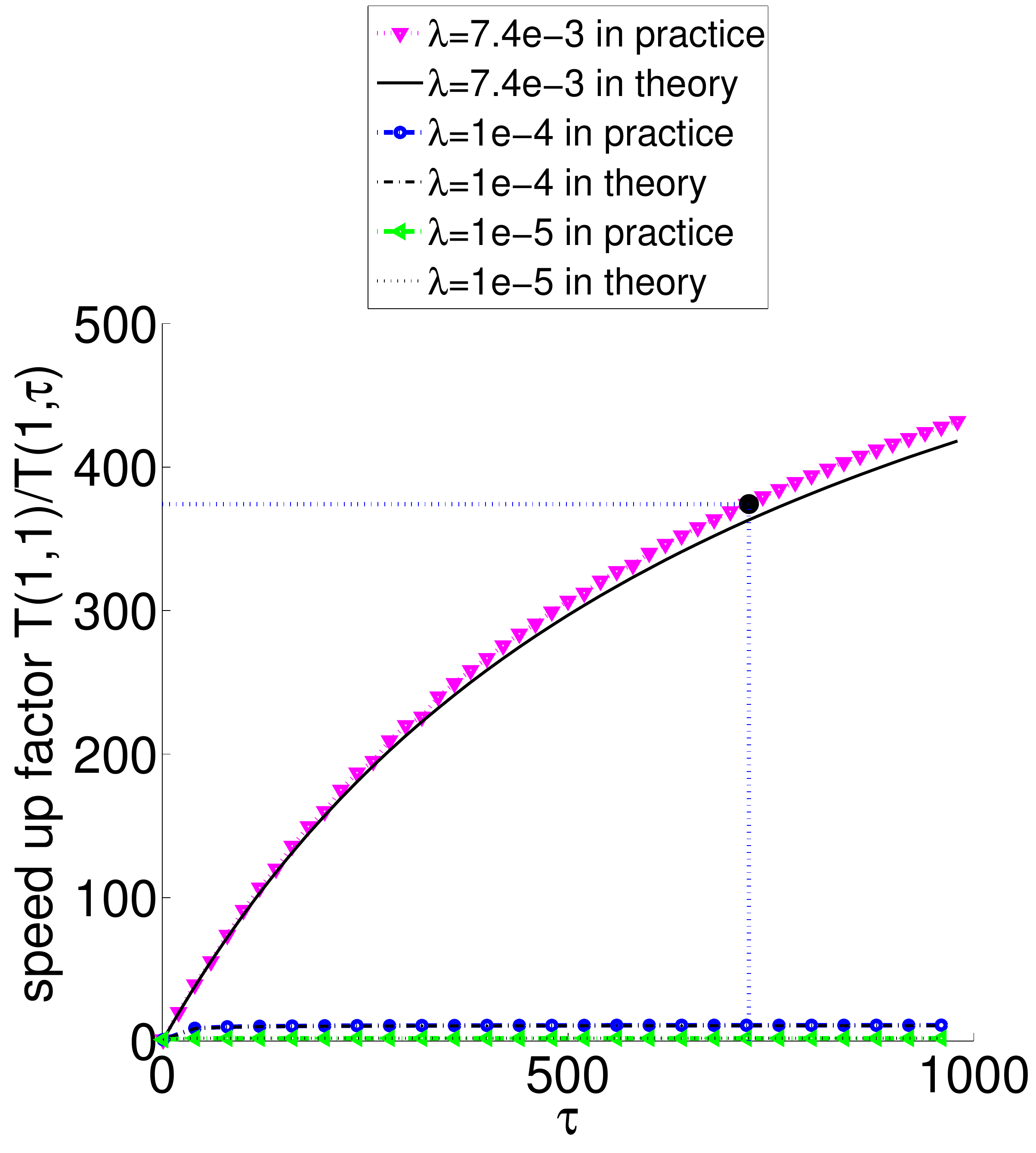}
    \label{figtau:subfig2}
}

\caption[Optional caption for list of figures]{Plots of theoretical and practical speedup factors as a function of $\tau$, for selected  values of $\lambda$. Problem: $L2$-regularized linear SVM with smoothed hinge loss and $\sigma=1$. Datasets: astro\_ph has $n=29,882$  training samples,
CCAT has $n=781,265$ training samples and cov1 has $n=522,911$ training samples. }
\label{fig:batch}
\end{figure}

In Figure~\ref{fig:batch} we  plot the speedup factors  for Quartz specialized to the $\tau$-nice sampling
on three different datasets: astro\_ph, CCAT and cov1, and for several  values of $\lambda$. We observe  that the practical speedup factor follows the theoretical prediction. Moreover, note that the largest $\lambda$ that we choose for each dataset is to have roughly
$$
\frac{\lambda \gamma n}{\max_i A_i^\top A_i} =\sqrt{n},
$$
so that linear speedup is reached for all $\tau \leq \sqrt{n}$, regardless of   data sparsity.

In Figure~\ref{figlog:ctau} we present contour lines of the theoretical and practical speedup factors, for Quartz specialized to the $(c,\tau)$-nice sampling on the  webspam dataset. We believe it is remarkable that the theoretical predictions are so accurate. Moreover, recall from the discussion in Section~\ref{sec:distr-speedup} that $T(c,\tau)$ is almost constant along the
contour lines of $c\tau$; this is why we see nearly straight lines in the log-log plot. This feature  is  observed here for the real dataset also.

%

\begin{figure}[ht]
\centering
\subfigure[theoretical speedup factor. data: webspam; $n$=350,000; sparsity: 33.51\% ]{
    \includegraphics[width=0.4\textwidth]{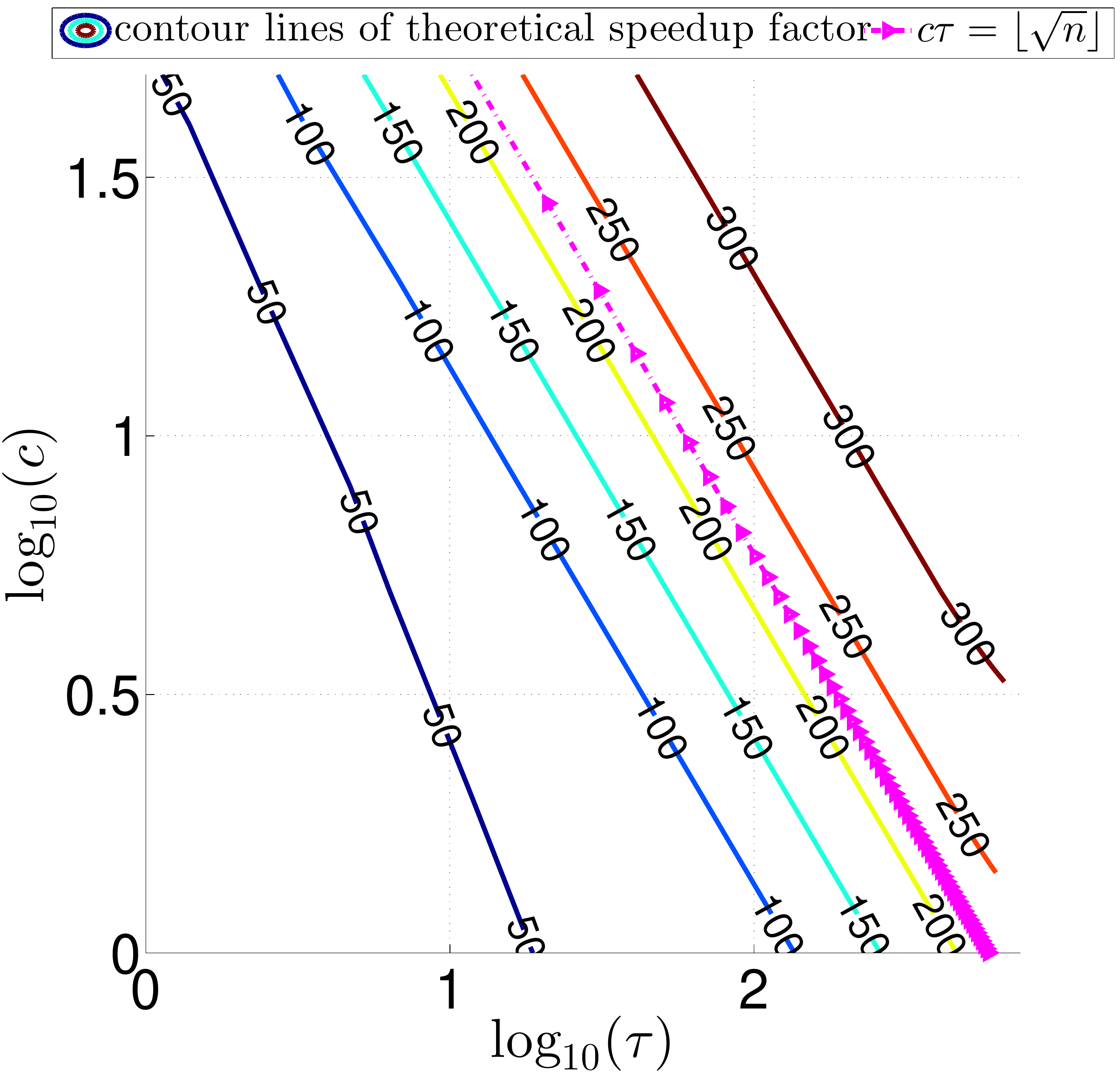}
    \label{figclog:subfig1}
}
\subfigure[experimental speedup factor. data: webspam; $n$=350,000; sparsity: 33.51\%]{
     \includegraphics[width=0.4\textwidth]{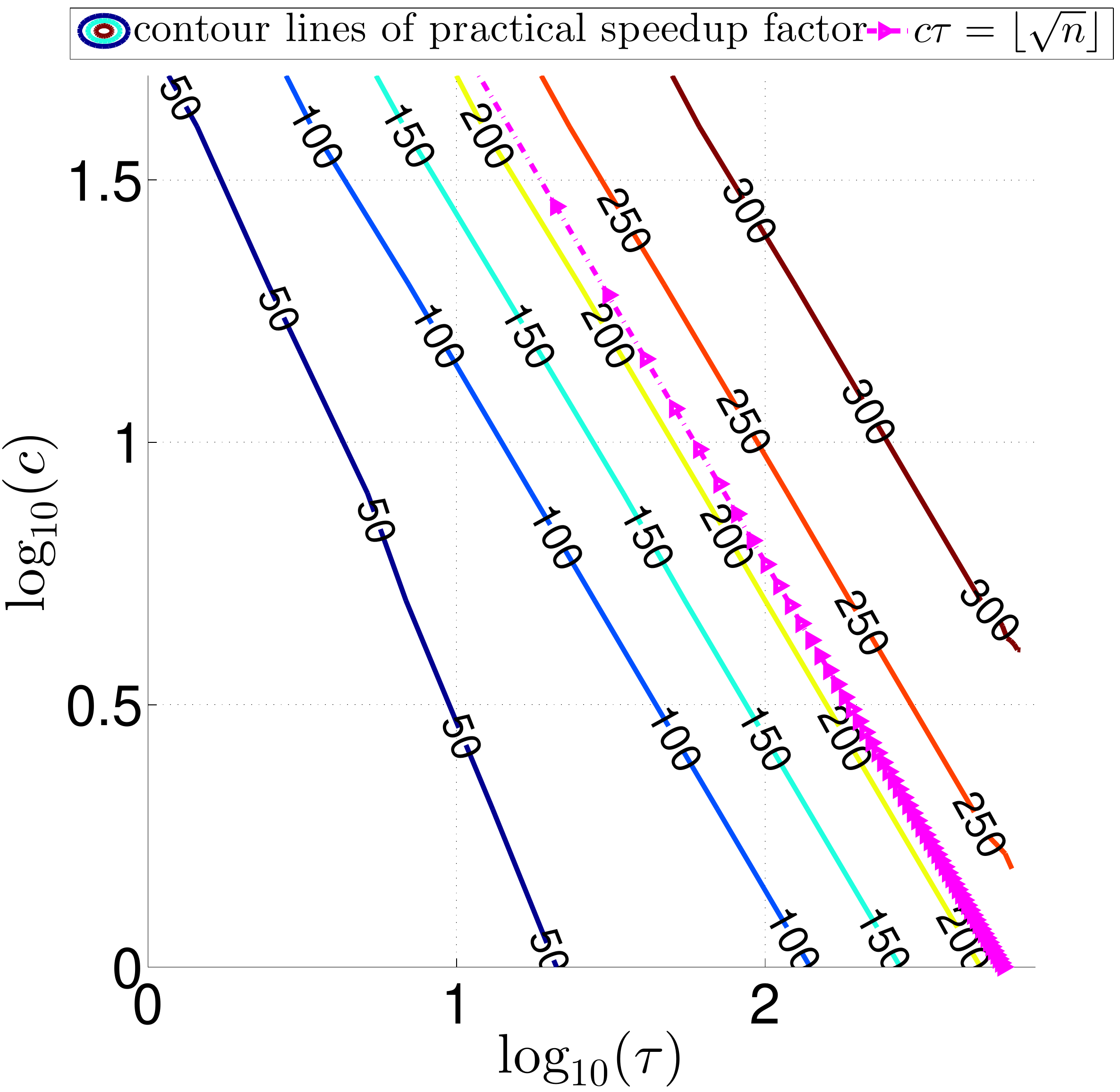}
    \label{figclog:subfig3}
}

\caption[Optional caption for list of figures]{Contour plots of theoretical~(Figure~\ref{figclog:subfig1}) and practical (Figure~\ref{figclog:subfig3}) speedup 
factor. The data set used is webspam. The loss function
used in the smoothed hinge loss with $\sigma=1$. }
\label{figlog:ctau}
\end{figure}

 \clearpage
\section{Conclusion}\label{sec:conclusion}

In this paper we have developed and analyzed a novel stochastic primal-dual algorithm---Quartz---for solving problems \eqref{eq:primal} and \eqref{eq:dual}. This is the second stochastic method which allows an arbitrary sampling (see \cite{NSync}) and the first primal-dual stochastic method with arbitrary sampling. 
This flexibility allows for many interesting variants of Quartz, including serial, parallel and distributed versions. The distributed variant of Quartz is the first distributed SDCA-like method with strong theoretical convergence bounds.

In Table~\ref{tbl:finalconclusion} we highlight selected characteristics of existing primal-dual stochastic methods. 

\begin{table}[ht]
\footnotesize
\centering
 \begin{tabular}{ >{\centering\arraybackslash}m{0.18\textwidth}| | >{\centering\arraybackslash}m{0.5in} |>{\centering\arraybackslash}m{0.7in} | >{\centering\arraybackslash}m{0.5in} | >{\centering\arraybackslash}m{0.5in} || >{\centering\arraybackslash}m{0.6in} | >{\centering\arraybackslash}m{0.6in} | >{\centering\arraybackslash}m{0.5in}  }
Algorithm & Serial uniform sampling & Serial optimal (importance) sampling &  $\tau$-nice sampling &  Arbitrary sampling  & Additional data-dependent speedup & Direct primal-dual analysis & Acceleration\\
\hline
\hline
SDCA \cite{SDCA}   &\cmark  &\xmark &\xmark & \xmark &\xmark & \xmark & \xmark\\
\hline
ASDCA \cite{ASDCA}  &\cmark &\xmark & \cmark & \xmark &\xmark & \xmark &\cmark \\
\hline
AccProx-SDCA \cite{ASDCA}  &\cmark & \xmark & \xmark & \xmark & \xmark &\xmark  &\cmark\\
\hline
DisDCA \cite{NIPSdistributedSDCA}  &\cmark & \xmark  &\cmark & \xmark &\xmark & \xmark & \xmark \\
\hline
Iprox-SDCA \cite{IProx-SDCA}  & \cmark &  \cmark & \xmark & \xmark &\xmark & \xmark & \xmark\\
\hline
APCG \cite{APCG} &\cmark & \xmark & \xmark & \xmark & \xmark &\xmark &\cmark \\
\hline 
SPDC \cite{SPDC}  &\cmark & \cmark &\cmark & \xmark & \xmark &  \cmark &\cmark \\
\hline
\hline
\bf{Quartz}  &\cmark  &\cmark &\cmark & \cmark & \cmark & \cmark & \xmark\\

 \end{tabular}
 \caption{Summary of selected characteristics of stochastic primal-dual algorithms.}
 \label{tbl:finalconclusion}
\end{table}

Unlike some of the existing methods, our method is not accelerated. We leave the development of an accelerated Quartz method for future research.

\bibliographystyle{plain}
\bibliography{literature}

\clearpage

\end{document}